\documentclass[a4paper]{amsart}
\usepackage{amssymb}
\usepackage{mathtools}
\usepackage[hidelinks]{hyperref}
\usepackage[numbered]{bookmark}     
\usepackage{enumitem}               
\usepackage{braket}                 
\usepackage{todonotes}              
\usepackage{multicol}
\usepackage{tikz-cd}

\usepackage[maxbibnames=10, 
  url=false, 
  doi=false, 
  date=year,
  isbn=false]{biblatex}
\DeclareFieldFormat{postnote}{#1}
\DeclareFieldFormat{multipostnote}{#1}

\AtEveryBibitem{\clearlist{language}} 

\newcommand{\C}{\mathbb{C}}
\newcommand{\N}{\mathbb{N}}

\newtheorem{thm}{Theorem}[section]
\newtheorem{corollary}[thm]{Corollary}
\newtheorem{definition}[thm]{Definition}
\newtheorem{example}[thm]{Example}
\newtheorem{lemma}[thm]{Lemma}
\newtheorem{proposition}[thm]{Proposition}
\newtheorem{remark}[thm]{Remark}

\newtheorem*{proposition*}{Proposition}
\newtheorem{notation}[thm]{Notation}
\newtheorem*{lemma*}{Lemma}
\newtheorem*{definition*}{Definition}
\newtheorem{problem}[thm]{Problem}

\newtheorem{introthm}{Theorem}

\newtheorem{introprop}[introthm]{Proposition}

\title[Hypercube C*-Algebras]{Hypercube C*-Algebras and an Application to Magic Isometries}
\date{\today}

\author{Bj\"orn Sch\"afer}
\email{bschaefer@math.uni-sb.de}
\address{Department of Mathematics, Saarland University, Postbox 151150, D-66041 Saarbr\"ucken, Germany}

\begin{document}


\begin{abstract}
    We study C*-algebras generated by two partitions of unity with orthogonality relations governed by hypercubes $Q_n$ for $n \in \N \setminus \{0\}$. These ``hypercube C*-algebras'' are special cases of bipartite graph C*-algebras which have been investigated by the author in a previous work. 
    We prove that the hypercube C*-algebras $C^\ast(Q_n)$ are subhomogeneous and obtain an explicit description as algebra of continuous functions from a standard simplex into a finite-dimensional matrix algebra with suitable boundary conditions. 
    Thus, we generalize Pedersen's description of the universal unital C*-algebra $C^\ast(p,q)$ of two projections. 
    We use our results to prove that any $2 \times 4$ ``magic isometry'' matrix can be filled up to a $4 \times 4$ ``magic unitary'' matrix. This answers a question from Banica, Skalski and So\l tan.
\end{abstract}

\maketitle

\setcounter{tocdepth}{1}    
\bookmarksetup{depth=2}     
\tableofcontents

\section{Introduction}
\label{sec::introduction}

Consider the universal unital C*-algebra $C^*(p,q)$ generated by two projections $p$ and $q$. It has been first noted by Pedersen \cite{pedersen_measure_1968} that this C*-algebra can be described very explicitly as algebra of continuous functions from the interval $[0,1]$ into the $2 \times 2$-matrices with suitable boundary conditions. More precisely, one has
\begin{align}
  \label{intro::eq:C*(p,q)_as_continuous_functions}
  C^\ast(p,q) \cong \{ f \in C([0,1], M_2) \mid f(0), f(1) \text{ are diagonal} \},
\end{align}
and the isomorphism is given by 
\begin{align*}
  p \mapsto \begin{pmatrix}
    1 & 0 \\
    0 & 0
  \end{pmatrix}_{t \in [0,1]}, \quad
  q \mapsto \begin{pmatrix}
    t & \sqrt{t(1-t)} \\
    \sqrt{t(1-t)} & 1-t
  \end{pmatrix}_{t \in [0,1]}.
\end{align*}
One can prove this theorem by different methods: For instance it follows from Halmos' description of two projections in generic position \cite{halmos_two_1969}, or it can be proved using the Mackey machine \cite{raeburn_c-algebra_1989}. More generally, Roch and Silbermann studied Banach algebras generated by two idempotents in \cite{roch_algebras_1988} and recovered the same result by very different methods. In fact, they proved that any unital Banach algebra $A$ generated by two idempotents satisfies the \emph{standard identity} of order $4$ which means that for any four elements $x_1, x_2, x_3, x_4 \in A$ the identity
\begin{align*}
  \sum_{\sigma \in S_4} \mathrm{sgn}(\sigma) x_{\sigma(1)} x_{\sigma(2)} x_{\sigma(3)} x_{\sigma(4)} = 0
\end{align*}
is true where $S_4$ is the symmetric group on four elements and $\mathrm{sgn}(\sigma)$ is the sign of the permutation $\sigma$. By general results, see e.g. \cite{krupnik_banach_1987}, it follows that for every maximal ideal $I$ of $A$ the quotient $A/I$ is isomorphic to some matrix algebra $M_\ell$ with $\ell \leq 2$. This enables a detailed description of $A$, and one recovers Pedersen's result in the case where $A$ is the universal C*-algebra generated by two projections.

In the preceding work \cite{schafer_classification_2025} the author investigated a class of C*-algebras which generalizes the algebra $C^\ast(p,q)$. Let us recall the definition.

\begin{definition}
  Given a bipartite graph $G = (U, V, E)$ with disjoint vertex sets $U$ and $V$ and edge set $E \subset \{\{u,v\} \mid u \in U, v \in V\}$ the \emph{bipartite graph C*-algebra} $C^\ast(G)$ is defined as the universal unital C*-algebra generated by projections $p_x$ for $x \in U \cup V$ subject to the relations
  \begin{align}
      \sum_{u \in U} p_u &= 1 = \sum_{v \in V} p_v,        \tag{GP1} \\
      p_u p_v &= 0 \; \text{ if } \{u, v\} \not \in E.     \tag{GP2}
  \end{align}
\end{definition}
These C*-algebras occur naturally when one studies the hypergraph C*-algebras of Trieb, Weber and Zenner \cite{trieb_hypergraph_2024}. An open question regarding those C*-algebras is for which hypergraphs $\mathrm{H}\Gamma$ the hypergraph C*-algebra $C^\ast(\mathrm{H}\Gamma)$ is nuclear. The author proved in \cite{schafer_nuclearity_2024} that it suffices to know which bipartite graph C*-algebras are nuclear in order to tell which hypergraph C*-algebras are nuclear, see also the discussion in \cite{schafer_classification_2025}.

It is an easy observation that a bipartite graph C*-algebra $C^\ast(G)$ is not nuclear if $K_{2,3} \subset G$, where $K_{2,3}$ is the complete bipartite graph with two vertices on one side and three vertices on the other side, see \cite[Corollary 2.9]{schafer_classification_2025}. However, the inverse implication is not known, and the following remains an open problem.

\begin{problem}
  \label{intro::problem:nuclearity_implied_by_K23_not_subgraph}
  Is $C^\ast(G)$ nuclear whenever $K_{2,3} \not \subset G$ holds?
\end{problem}

In this paper, we investigate the bipartite graph C*-algebras $C^\ast(Q_n)$ associated to the hypercubes $Q_n$ of dimension $n \in \N \setminus \{0\}$ where the hypercubes are naturally considered as bipartite graphs. Since $K_{2,3} \not \subset Q_n$ this allows to investigate Problem \ref{intro::problem:nuclearity_implied_by_K23_not_subgraph} in a particular, more manageable situation. The first three hypercubes $Q_1, Q_2, Q_3$ are sketched in Figure \ref{int::fig:hypercubes}. To study $C^\ast(Q_n)$ we use a method inspired by Roch and Silbermann \cite{roch_algebras_1988} whose approach was outlined above. In fact, we exploit the particular combinatorial structure of the hypercube to show that every corner $p_x C^\ast(Q_n) p_x$ for a vertex $x$ of $Q_n$ is commutative. As a result, the following theorem is obtained. Recall that a C*-algebra is called $k$-subhomogeneous if the dimension of the underlying Hilbert space of every irreducible representation is at most $k$.

\begin{figure}[htbp]
    \centering
    \begin{tikzpicture}
        \begin{scope}[xshift=0cm]
            \coordinate (A) at (0,1);
            \coordinate (B) at (2,1);
            
            \draw[thick] (A) -- (B);
            
            \filldraw[fill=white, draw=black] (A) circle (2pt);
            \filldraw[black] (B) circle (2pt);
            
            \node at (1,-0.4) {$Q_1$};
        \end{scope}
        
        \begin{scope}[xshift=4cm]
            \coordinate (A) at (0,0);
            \coordinate (B) at (1.5,0);
            \coordinate (C) at (1.5,1.5);
            \coordinate (D) at (0,1.5);
            
            \draw[thick] (A) -- (B) -- (C) -- (D) -- cycle;
            
            \filldraw[fill=white, draw=black] (A) circle (2pt);
            \filldraw[black] (B) circle (2pt);
            \filldraw[fill=white, draw=black] (C) circle (2pt);
            \filldraw[black] (D) circle (2pt);
            
            \node at (0.75,-0.4) {$Q_2$};
        \end{scope}
        
        \begin{scope}[xshift=8cm]
            \coordinate (A) at (0,0);
            \coordinate (B) at (1.5,0);
            \coordinate (C) at (1.5,1.5);
            \coordinate (D) at (0,1.5);
            \coordinate (E) at (0.5,0.5);
            \coordinate (F) at (2,0.5);
            \coordinate (G) at (2,2);
            \coordinate (H) at (0.5,2);
            
            \draw[thick] (A) -- (B) -- (C) -- (D) -- cycle;
            \draw[thick] (E) -- (F) -- (G) -- (H) -- cycle;
            \draw[thick] (A) -- (E);
            \draw[thick] (B) -- (F);
            \draw[thick] (C) -- (G);
            \draw[thick] (D) -- (H);
            
            \filldraw[fill=white, draw=black] (A) circle (2pt);
            \filldraw[black] (B) circle (2pt);
            \filldraw[fill=white, draw=black] (C) circle (2pt);
            \filldraw[black] (D) circle (2pt);
            \filldraw[black] (E) circle (2pt);
            \filldraw[fill=white, draw=black] (F) circle (2pt);
            \filldraw[black] (G) circle (2pt);
            \filldraw[fill=white, draw=black] (H) circle (2pt);
            
            \node at (1,-0.4) {$Q_3$};
        \end{scope}
    \end{tikzpicture}
    \caption{The hypercubes $Q_1$, $Q_2$ and $Q_3$}
    \label{int::fig:hypercubes}
\end{figure}
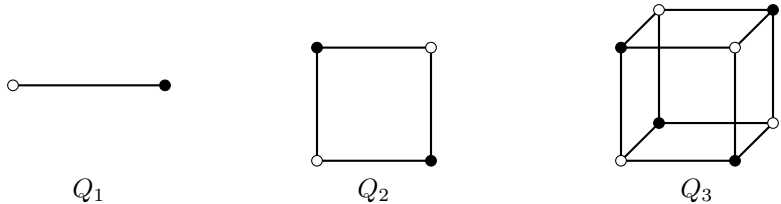

\begin{introthm}[{Theorem \ref{hyp::thm:subhomogeneous:hypercube_algebra_is_subhomogeneous}}]
    The $C^\ast$-algebra $C^\ast(Q_n)$ is $2^{n-1}$-subhomogeneous, and one has for all vertices $x$ of $Q_n$ and for every irreducible representation $\rho$ of $C^\ast(Q_n)$ on a Hilbert space $\mathcal{H}$,
    \begin{align*}
        \mathrm{dim}(\rho(p_x) \mathcal{H}) \leq 1.
    \end{align*}
    We call such a representation \emph{rank-one}.
\end{introthm}

In particular, it follows that $C^\ast(Q_n)$ is nuclear as a $2^{n-1}$-subhomogeneous C*-algebra. This gives an affirmative answer to Problem \ref{intro::problem:nuclearity_implied_by_K23_not_subgraph} for the subclass $\{Q_n \mid n \in \N \setminus \{0\}\} \subset \{G \mid K_{2,3} \not \subset G\}$ of bipartite graphs with $K_{2,3} \not \subset Q_n$.

What's more, we obtain an explicit description of $C^\ast(Q_n)$ as algebra of continuous functions from the $(n-1)$-dimensional standard simplex 
$$
  \Delta_{n-1} = \left\{[t_0, \dots, t_{n-1}] \mid t_0, \dots, t_{n-1} \geq 0, \, \sum_i t_i = 1\right\}
$$ 
into the $2^{n-1} \times 2^{n-1}$-matrices $M_{2^{n-1}}$ with suitable boundary conditions. For the definition of $\mathbf{t}$-block diagonal form in the following theorem see the beginning of Section \ref{sec::explicit_description_of_hypercube_algebras}.

\begin{introthm}[Theorem \ref{hyp::thm:hypercube_cstar_algebras_as_continuous_functions}]
    The C*-algebra $C^\ast(Q_n)$ is isomorphic to
    \begin{align*}
        A_n := \left\{ f \in C(\Delta_{n-1}, M_{2^{n-1}}) \mid f(\mathbf{t}) \text{ is in } \mathbf{t}\text{-block diagonal form for all } \mathbf{t} \in \Delta_{n-1} \right\}.
    \end{align*}
\end{introthm}
For $n=2$ it is $C^\ast(Q_n) = C^\ast(Q_2) \cong C^\ast(p,q)$, $\Delta_{n-1} = [0,1]$ and $M_{2^{n-1}} = M_2$, and one recovers Pedersen's description of $C^\ast(p,q)$ as algebra of continuous functions from $[0,1]$ into $M_2$ from \eqref{intro::eq:C*(p,q)_as_continuous_functions}.

As an application, we use our findings to analyze $2 \times 4$ magic isometries. Let us first recall their definition from e.g. \cite[Definition 5.1]{banica_noncommutative_2012}.

\begin{definition} 
    \label{intro::def:magic_isometry_unitary}
    Let $m \leq n$ and let $A$ be a unital $C^\ast$-algebra. An $m \times n$-matrix $P = (P_{ij})_{i \leq m, j \leq n} \in M_{m,n}(A)$ with entries from $A$ is called a \emph{magic isometry} if the entries $P_{ij}$ are projections which form a partition of unity along rows and are pairwise orthogonal along columns, i.e.
    \begin{align*}
        (P_{ij})_{j=1, \dots, n} &\text{ is a partition of unity for each fixed } i \leq m, \\
        (P_{ij})_{i=1, \dots, m} &\text{ is a family of pairwise orthogonal projections for each fixed } j \leq n.
    \end{align*} 
    Here, with a partition of unity we mean a family of projections in a unital $C^\ast$-algebra that sum up to the unit. 
    
    If $m = n$ and the projections $P_{ij}$ form a partition of unity along both rows and columns, then we call $P$ a \emph{magic unitary}.
\end{definition}

The following problem was first raised by Banica, Skalski and So\l tan in \cite[Section 5]{banica_noncommutative_2012}, see also Weber's introductory article \cite[Problem 3.4]{weber_quantum_2023}.

\begin{problem}
    \label{intro::problem:magic_isometry_fill_up}
    Can every $m \times n$-magic isometry be filled up to a magic unitary matrix, i.e. given a magic isometry 
    $$
    P = (P_{ij})_{i \leq m, j \leq n} \in M_{m,n}(A),
    $$ 
    where $A = C^\ast(P_{ij})$ and $m \leq n$, can one find another $C^\ast$-algebra $B$ and a magic unitary 
    $$
    U = (U_{ij})_{i,j \leq n} \in M_n(B)
    $$ 
    such that $A$ embeds into $B$ via $P_{ij} \mapsto U_{ij}$ for $i \leq m, j \leq n$?
\end{problem}

It turns out that the universal $C^\ast$-algebra generated by the entries of a $2 \times 4$-magic isometry is the bipartite graph algebra $C^\ast(Q_3)$ associated to the three-dimensional cube. On the other hand, a useful matrix model for the universal $C^\ast$-algebra generated by the entries of  $4 \times 4$-magic unitary, $C(S_4^+)$, has been described by Banica and Collins in \cite{banica_integration_2008}. As already suggested by Banica, Skalski and So\l tan we can use their matrix model and combine it with our description of $C^\ast(Q_3)$ to solve the above problem in the affirmative for $m=2$ and $n=4$.

\begin{introprop}[Proposition \ref{hyp::prop:quantum_sudoku_fill_up}]
    Every $2 \times 4$-magic isometry can be filled up to a $4 \times 4$-magic unitary matrix in the sense of Problem \ref{intro::problem:magic_isometry_fill_up}.
\end{introprop}

\subsection{Outline}

In \textbf{Section \ref{sec::preliminaries}} we collect some basic definitions about bipartite graph C*-algebras and hypercubes. In \textbf{Section \ref{sec::subhomogeneous}} we prove that every hypercube C*-algebra $C^\ast(Q_n)$ is subhomogeneous by analyzing loops in the hypercube. In \textbf{Section \ref{sec::hypercube_rank_one_representations}} we describe rank-one representations of $C^\ast(Q_n)$ using complex edge weightings of $Q_n$. We rely on these edge weightings to classify all irreducible representations of hypercube C*-algebras in \textbf{Section \ref{sec::irreducible_representations}}. Using this it is not difficult to describe $C^\ast(Q_n)$ as algebra of continuous functions in \textbf{Section \ref{sec::explicit_description_of_hypercube_algebras}}. Finally, in \textbf{Section \ref{sec::application_to_magic_isometries}} we apply our results to prove that any $2 \times 4$-magic isometry matrix can be filled up to a $4 \times 4$-magic unitary matrix.

\subsection{Acknowledgements}

The author is grateful to his supervisor Moritz Weber for his support and helpful advices. Furthermore, the author would like to thank Steffen Roch for making him aware of and discussing his results on Banach algebras generated by two idempotents and pointing out other literature. 

This work is part of the author's PhD thesis and a contribution to the SFB-TRR 195. During the time of writing, the author co-supervised the master's thesis of Fabian Selzer, who worked on a similar topic as in Section 4 of the present article, and who presented an alternative proof of Theorem \ref{hyp::thm:hypercube_cstar_algebras_as_continuous_functions}. 

\section{Preliminaries}
\label{sec::preliminaries}

In this section we collect some basic definitions about bipartite graph C*-algebras and hypercubes. Definitions and first observations regarding bipartite graph C*-algebras are copied verbatim from \cite{schafer_classification_2025}.

\subsection{Bipartite graph C*-algebras}

Bipartite graph C*-algebras have been introduced by the author in a previous paper. We recall their definition shortly and refer to \cite{schafer_classification_2025} for more details. 

\begin{definition}
    A bipartite graph $G$ consists of two vertex sets $U$ and $V$ together with an edge set $E \subset \{\{u,v\}\mid u \in U, v \in V\}$. We write $G = (U, V, E)$.
\end{definition}

For two vertices in a bipartite graph $G = (U, V, E)$ let us write $u \sim v$ if $\{u,v\} \in E$ and let $\mathcal{N}(u) := \{v \in V\mid v \sim u\}$ ($\mathcal{N}(v) := \{u \in U\mid v \sim u\}$) be the neighbors of $u$ ($v$). For vertices $x \in U \cup V$ or edges $e \in E$ we also write $x \in G$ or $e \in G$, respectively, and say that $x$ or $e$ are contained in $G$. 

A path $\mu$ in a bipartite graph $G$ is a finite sequence of vertices $x_1 \dots x_n$ such that $x_i \sim x_{i+1}$ holds for all $i < n$. For every $i < n$, we say that $x_i$ is contained in $\mu$. We set $s(\mu) = x_1$ and $r(\mu) = x_n$. For two paths $\mu=x_1 \dots x_m$ and $\nu = y_1 \dots y_n$ with $r(\mu) = s(\nu)$ we write $\mu \nu$ for the combined path $x_1 \dots x_m y_1 \dots y_n$.

Subgraphs and induced subgraphs of a bipartite graph are defined in the natural way, for more details see \cite[Definition 2.3]{schafer_classification_2025}. Examples of bipartite graphs include the complete bipartite graphs $K_{m,n}$ for $m,n \in \N$.

Let us next recall the definition of bipartite graph C*-algebras.

\begin{definition}[{\cite[Definition 2.4]{schafer_classification_2025}}]
    \label{bipartite_graph_algebras:definition}
    Given a bipartite graph $G = (U, V, E)$ let $C^\ast(G)$ be the universal $C^\ast$-algebra generated by a family of projections $(p_x)_{x \in U \cup V}$ subject to the following relations:
    \begin{align}
        \sum_{u \in U} p_u &= 1 = \sum_{v \in V} p_v,        \tag{GP1} \label{bga::eq:partition_relation} \\
        p_u p_v &= 0 \; \text{ if } \{u, v\} \not \in E.  \tag{GP2} \label{bga::eq:orthogonality_relation}
    \end{align}
\end{definition}

The following dense subset of $C^\ast(G)$ is immediate.

\begin{proposition}[{\cite[Proposition 2.6]{schafer_classification_2025}}]
    \label{prop::bipartite_graph_algebras:definition:dense_subset}
    Let $G$ be a bipartite graph and for every path $\mu = x_1 \dots x_n$ in $G$, write $p_\mu := p_{x_1} \cdots p_{x_n}$ for the associated element in $C^\ast(G)$. Then the elements $p_\mu$ span a dense subset of $C^\ast(G)$. 
\end{proposition}

\subsection{Hypercubes}
\label{sec::hypercubes}

In this section we give a precise definition of the hypercubes $Q_n$ for $n \in \N \setminus \{0\}$ using binary numbers. We also collect some properties of hypercubes that will be useful later. 

The following notation will be crucial throughout this paper.

\begin{notation}
    \label{pre::notation:binary_numbers}
    For every number $i \in \N$ and $n \in \N$ such that $i < 2^n$ the $n$-digit binary representation of $i$ is the sequence 
    \begin{align*}
        i_{n-1} \dots i_1 i_0 \in \{0,1\}^n \quad \text{ such that } \quad i = \sum_{k=0}^{n-1} i_k 2^{k}.
    \end{align*}
    We call the number $i_k$ the $k$-th binary digit of $i$ and denote $[i]_k := i_k$. Further, we set
    \begin{align*}
        \mathrm{par}_k(i) = \sum_{\ell=0}^{k} i_\ell \mod 2.
    \end{align*}
    Note that $\mathrm{par}_0(i)$ is the parity of the number $i$.

    Further, for every $k < n$ let $i \# k \in \N$ be the number obtained by flipping the $k$-th binary digit of $i$, i.e. 
    \begin{align*}
        i \# k = \begin{cases}
            i + 2^{k} & \text{if } i_k = 0, \\
            i - 2^{k} & \text{if } i_k = 1.
        \end{cases}
    \end{align*}
\end{notation}

Note that we will be using square brackets $[$ $]$ both to denote digits of a binary number $[i]_\ell$ and to indicate components of a vector $[\psi]_i$. It is always clear from the context which meaning is intended.

\begin{definition}
    The $n$-dimensional hypercube $Q_n = (U_n, V_n, E_n)$ is the bipartite graph with vertex sets
    \begin{align*}
        U_n = \{i < 2^n \mid \mathrm{par}_{n-1}(i) = 0\}, \qquad
        V_n = \{j < 2^n \mid \mathrm{par}_{n-1}(j) = 1\}
    \end{align*}
    and edge set
    \begin{align*}
        E_{n} = \{ \{i,j\} \mid i \in U_n, \; j \in V_n, \; i \# k = j \text{ for some } k < n\}.
    \end{align*}
    We denote the edge $\{i,j\} \in E_{n}$ with $i \in U_n$ and $j \in V_n$ also by $ij \in E_{n}$.
\end{definition}

\begin{example}
    Let us take a second look at the three hypercubes $Q_1, Q_2, Q_3$ that were sketched in the introduction in Figure \ref{int::fig:hypercubes}. We can label their vertices by binary numbers as in Figure \ref{pre::fig:hypercubes_with_binary_labels}.
    \begin{figure}[htbp]
    \centering
    \begin{tikzpicture}
            \begin{scope}[xshift=0cm]
                \coordinate (A) at (0,1);
                \coordinate (B) at (2,1);

                \draw[thick] (A) -- (B);

                \filldraw[fill=white, draw=black] (A) circle (2pt);
                \filldraw[black] (B) circle (2pt);

                \node[left] at (A) {$0$};
                \node[right] at (B) {$1$};

                \node at (1,-0.4) {$Q_1$};
            \end{scope}

            \begin{scope}[xshift=4cm]
                \coordinate (A) at (0,0);
                \coordinate (B) at (1.5,0);
                \coordinate (C) at (1.5,1.5);
                \coordinate (D) at (0,1.5);

                \draw[thick] (A) -- (B) -- (C) -- (D) -- cycle;

                \filldraw[fill=white, draw=black] (A) circle (2pt);
                \filldraw[black] (B) circle (2pt);
                \filldraw[fill=white, draw=black] (C) circle (2pt);
                \filldraw[black] (D) circle (2pt);

                \node[below left] at (A) {$00$};
                \node[below right] at (B) {$10$};
                \node[above right] at (C) {$11$};
                \node[above left] at (D) {$01$};

                \node at (0.75,-0.4) {$Q_2$};
            \end{scope}

            \begin{scope}[xshift=8cm]
                \coordinate (A) at (0,0);
                \coordinate (B) at (1.5,0);
                \coordinate (C) at (1.5,1.5);
                \coordinate (D) at (0,1.5);
                \coordinate (E) at (0.5,0.5);
                \coordinate (F) at (2,0.5);
                \coordinate (G) at (2,2);
                \coordinate (H) at (0.5,2);

                \draw[thick] (A) -- (B) -- (C) -- (D) -- cycle;
                \draw[thick] (E) -- (F) -- (G) -- (H) -- cycle;
                \draw[thick] (A) -- (E);
                \draw[thick] (B) -- (F);
                \draw[thick] (C) -- (G);
                \draw[thick] (D) -- (H);

                \filldraw[fill=white, draw=black] (A) circle (2pt);
                \filldraw[black] (B) circle (2pt);
                \filldraw[fill=white, draw=black] (C) circle (2pt);
                \filldraw[black] (D) circle (2pt);
                \filldraw[black] (E) circle (2pt);
                \filldraw[fill=white, draw=black] (F) circle (2pt);
                \filldraw[black] (G) circle (2pt);
                \filldraw[fill=white, draw=black] (H) circle (2pt);

                \node[below left] at (A) {$000$};
                \node[below right] at (B) {$100$};
                \node[above left] at (C) {$110$};
                \node[above left] at (D) {$010$};
                \node[below right] at (E) {$001$};
                \node[below right] at (F) {$101$};
                \node[above right] at (G) {$111$};
                \node[above left] at (H) {$011$};

                \node at (1,-0.4) {$Q_3$};
            \end{scope}
        \end{tikzpicture}
        \caption{The hypercubes $Q_1$, $Q_2$ and $Q_3$ with vertices labeled by binary numbers}
        \label{pre::fig:hypercubes_with_binary_labels}
    \end{figure}
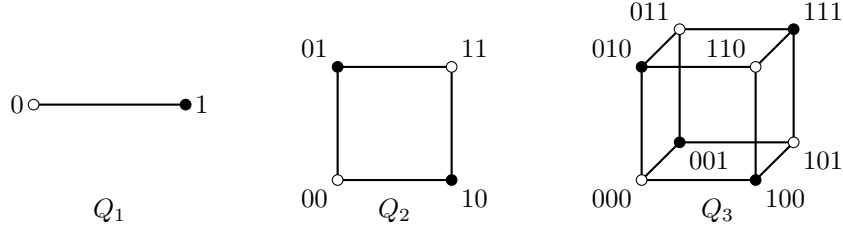
\end{example}

The following is an immediate observation.

\begin{proposition}
    \label{hyp::prop:common_neighbors_in_hypercubes}
    Two distinct vertices $x_1, x_2 \in U_n \cup V_n$ of the hypercube $Q_n$ have either zero or two common neighbors $y_1, y_2 \in U_n \cup V_n$. The latter is true if and only if $x_1$ and $x_2$ differ in exactly two binary digits. In this case the subgraph $Q_n(x_1, x_2, y_1, y_2) \subset Q_n$ induced by the vertices $x_1, x_2, y_1, y_2$ is isomorphic to the square $Q_2$.
    In particular, we have $K_{2,3} \not \subset Q_n$.
\end{proposition}

\begin{notation}
    \label{pre::notation:subhypercube_Q_n^(k)_and_E_n^(k)}
    Let $n \in \N \setminus \{0\}$ and $k < n$.
    \begin{enumerate}
        \item We define $Q_n^{(+k)} \subset Q_n$ as the subgraph induced by the vertices $x \in U_n \cup V_n$ with $[x]_k = 1$, and we let $Q_n^{(-k)}$ be the subgraph induced by the vertices $x \in U_n \cup V_n$ with $[x]_k = 0$.
        \item We let $E_n^{(k)} \subset E_n$ be the set of edges $ij \in E_n$ with $i = j \# k$.
    \end{enumerate}
\end{notation}

\begin{lemma}
    \label{hyp::lemma:properties_of_hypercubes2}
    The following properties hold for the hypercubes $Q_n$.
    \begin{enumerate}
        \item For every proper sub-hypercube $Q^\prime \subsetneq Q_n$ there is a $k < n$ such that $Q^\prime$ is contained in $Q_n^{(+k)}$ or $Q_n^{(-k)}$.
        \item Consider the equivalence relation on $E_n^{(k)}$ generated by 
        \begin{align*}
            i_1j_1 \sim_0 i_2j_2
            \quad :\Leftrightarrow \quad
            Q_n(i_1, i_2, j_1, j_2) \cong Q_2,
        \end{align*}
        where $Q_n(i_1, i_2, j_1, j_2) \subset Q_n$ is the subgraph induced by $i_1, i_2, j_1, j_2$.
        Then all edges $ij \in E_n^{(k)}$ are in the same equivalence class.
        \item Assume that a subgraph $G \subset Q_n$ is closed under squares in the sense that $G$ contains all edges of a square in $Q_n$ as soon as it contains two adjacent edges of the square. Then $G$ is a disjoint union of sub-hypercubes.
    \end{enumerate}
\end{lemma}

\begin{proof}
    Ad (1). Assume that $Q^\prime \cong Q_m$ for some $m < n$ and let $\varphi: \{0, \dots, 2^m-1\} \to \{0, \dots, 2^n-1\}$ be an embedding of $Q_m$ into $Q_n$. Without loss of generality $\varphi(0) = 0$. Then, for every $\ell < m$ there is some $k_\ell < n$ such that
    \begin{align*}
        \varphi(2^\ell) = 2^{k_\ell}
    \end{align*}
    holds. Indeed, $2^\ell$ is obtained from $0$ by flipping the $\ell$-th binary digit, and therefore $\varphi(2^\ell)$ must be obtained from $\varphi(0) = 0$ by flipping exactly one binary digit as well, for $\varphi$ respects the edge relation of the involved hypercubes. Clearly, the $k_\ell$ are pairwise distinct for different $\ell < m$. Now, one readily checks $Q^\prime \subset Q_n^{(-k)}$ for any $k \not \in \{k_\ell \mid \ell < m\}$.

    Ad (2). Phrasing the equivalence relation in other terms, one has $i_1 j_1 \sim_0 i_2 j_2$ if and only if there is some $k \neq \ell < n$ such that $j_2 = i_1 \# \ell$ and $i_2 = j_1 \# \ell$.
    Clearly, starting from some edge $i j \in E_n^{(k)}$ one can obtain any other edge $i^\prime j^\prime \in E_n^{(k)}$ by successively flipping binary digits of $i$ and $j$ other than the $k$-th digit (at the same time for both $i$ and $j$). Each such step corresponds to an application of $\sim_0$, and thus the statement follows.

    Ad (3). Without loss of generality assume that $G \subset Q_n$ is connected. Evidently, $G$ is a sub-hypercube of $Q_n$ if its vertex set $\{i < 2^n \mid i \in G\}$ is closed under flipping $k$-th binary digits where $k$ ranges over a set $K \subset \{0, \dots, n-1\}$ of indices. Let 
    $$
    K := \{k < n \mid \exists x \in G: x \# k \in G\},
    $$
    and let $i \in G$ be arbitrary. We need to show that $i \# k$ is in $G$ as well. By the definition of $K$ there exists some $j \in G$ with $j \# k \in G$, and since $G$ is connected there is a path $i_1 \dots i_m$ in $G$ which connects $i = i_1$ and $j = i_m$.
    
    The assumption from the statement tells us that $(x \# k ) \# \ell$ is in $G$ whenever $x \in G$ and both $x \# k$ and $x \# \ell$ are in $G$. We prove $i \# k \in G$ by induction over $m$. If $m = 1$ there is nothing to show. Otherwise, one has without loss of generality $i_m = i_{m-1} \# \ell$ for some $\ell \neq k$. It follows
    \begin{align*}
        i_{m-1} \# k
            &= (i_m \# \ell) \# k \in G,
    \end{align*}
    i.e. the vertex $i_{m-1}$ already has the property we required from $i_m$. Hence, the statement follows by induction.
\end{proof}

\section{Hypercube C*-algebras are subhomogeneous}
\label{sec::subhomogeneous}

In this section we prove that every hypercube C*-algebra $C^\ast(Q_n)$ is subhomogeneous, and more particularly, for every irreducible representation $\rho$ of $C^\ast(Q_n)$ and projection $p_x$ we have that $\rho(p_x)$ is either zero or a rank-one projection. We prove this by analyzing paths in the hypercube with the aim of showing
\begin{align*}
    p_\mu p_\nu = p_\nu p_\mu
\end{align*}
for any two loops $\mu$ and $\nu$ in $Q_n$.

To do that, let us describe a path $\mu$ in $Q_n$ in another way. Given $\mu = x_1 \dots x_{m+1}$ with vertices $x_1, \dots, x_{m+1} \in U_n \cup V_n$ one can choose indices $k_1, \dots, k_m < n$ such that
\begin{align*}
    x_{i+1} = x_i \# k_i
\end{align*}
holds for all $i \leq m$. In operational terms, the index $k_i$ records in which direction the path $\mu$ continues at the vertex $x_i$. It is possible to recover the whole path $\mu$ from the starting vertex $x_1$ and the sequence $k_1 \dots k_m$. Therefore, we may use the following notation.

\begin{notation}
    Given a vertex $x \in U_n \cup V_n$ and a sequence of indices $k_1, \dots, k_m < n$ we write
    \begin{align*}
        [x; \, k_1 \dots k_m]
    \end{align*}
    for the path $\mu = x_1 \dots x_{m+1}$ given by
    \begin{align*}
        x_1 &= x, &\text{ and } &&
        x_{i+1} &= x_i \# k_i \quad \text{for } i \leq m.
    \end{align*}
\end{notation}

\begin{proposition}
    \label{sub::prop:operation_on_paths}
    \begin{enumerate}
        \item For every path $\mu$ in $Q_n$ there is a unique vertex $x \in U_n \cup V_n$ and a unique sequence $k_1 \dots k_m$ of indices with $k_i < n$ such that
        \begin{align*}
            \mu = [x; \, k_1 \dots k_m].
        \end{align*}
        \item The path $[x; \, k_1 \dots k_m]$ is a loop if and only if for every $\ell < n$ the set $\{i \leq m \mid k_i = \ell\}$ has an even number of elements.
    \end{enumerate}
\end{proposition}

\begin{proof}
    The first statement follows from the definition of edges in $Q_n$. The second statement is true since the $\ell$-th binary digit of the endpoint of $[x; \, k_1 \dots k_m]$ is equal to $[x]_\ell + |\{i \leq m \mid k_i = \ell\}| \mod 2$.
\end{proof}

\begin{lemma}
    Let $\mu = [x; \, k_1 \dots k_i k_{i+1} \dots k_m]$ be a path in $Q_n$ as above, and obtain 
    $$
    \mu = [x; \, k_1 \dots k_{i-1} k_{i+1} k_i k_{i+2} \dots k_m]
    $$ 
    by swapping the elements $k_i$ and $k_{i+1}$ in the sequence $k_1 \dots k_m$ for some $i < m$.
    \begin{enumerate}
        \item If $k_i \neq k_{i+1}$, then $p_\mu = - p_\nu$.
        \item If $k_i = k_{i+1}$, then $p_\mu = p_\nu$.
    \end{enumerate}
\end{lemma}

\begin{proof}
    If $k_i = k_{i+1}$, then $\mu = \nu$ and the statement is trivial. So, let us assume $k_i \neq k_{i+1}$.
    Write $\mu = x_1 \dots x_{m+1}$ and $\nu = y_1 \dots y_{m^\prime + 1}$ for vertices $x_j, y_j \in U_n \cup V_n$. Clearly,
    \begin{align*}
        x_j = y_j \quad \text{ for all } j \leq i.
    \end{align*}
    Further, observe that 
    \begin{align*}
        x_{i+2} = (x_{i} \# k_i) \# k_{i+1} = (x_{i} \# k_{i+1}) \# k_i = (y_{i} \# k_{i+1}) \# k_i = y_{i+2},
    \end{align*}
    since it makes no difference in which order binary digits are flipped.
    Hence, we have as well
    \begin{align*}
        x_j = y_j \quad \text{ for all } j \geq i+2.
    \end{align*}
    Let us compare the elements $p_{x_{i}} p_{x_{i+1}} p_{x_{i+2}}$ and $p_{y_i} p_{y_{i+1}} p_{y_{i+2}}$. First, note that the terms differ only in the middle factor, for $x_i = y_i$ and $x_{i+2} = y_{i+2}$. One readily checks that the involved vertices form a subgraph of $Q_n$ isomorphic to the square as in Figure \ref{sub::fig:square_in_hypercube}.
    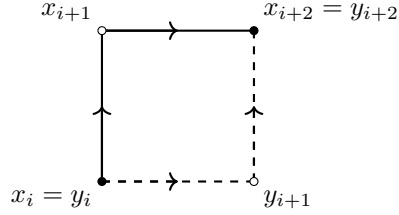
\begin{figure}[ht]
        \begin{tikzpicture}
            \coordinate (A) at (0, 0);
            \coordinate (B) at (2, 0);
            \coordinate (C) at (2, 2);
            \coordinate (D) at (0, 2);
            
            \draw [thick] (A) -- (D) -- (C) ;
            \draw[->, thick] (A) -- (0,1);
            \draw[->, thick] (D) -- (1,2);

            \draw[thick, dashed] (A) -- (B);
            \draw[thick, dashed] (B) -- (C);
            \draw[->, thick, dashed] (B) -- (2,1);
            \draw[->, thick, dashed] (A) -- (1,0);
            
            \filldraw[black] (A) circle (1.5pt) node[below left] {$x_i = y_i$}; 
            \filldraw[fill=white, draw=black] (B) circle (1.5pt) node[below right] {$y_{i+1}$}; 
            \filldraw[black] (C) circle (1.5pt) node[above right] {$x_{i+2} = y_{i+2}$}; 
            \filldraw[fill=white, draw=black] (D) circle (1.5pt) node[above left] {$x_{i+1}$};   
        \end{tikzpicture}
        \caption{The subgraph of $Q_n$ induced by $x_{i-1}, x_i, x_{i+1}, x_i^\prime$}
        \label{sub::fig:square_in_hypercube}
    \end{figure}

    In view of Proposition \ref{hyp::prop:common_neighbors_in_hypercubes} the vertices $x_{i+1}$ and $y_{i+1}$ are the only common neighbors of $x_i$ and $x_{i+1}$. Therefore, by \eqref{bga::eq:orthogonality_relation} we have
    \begin{align*}
        p_{x_i} p_{z} p_{x_{i+2}} = 0
    \end{align*}
    for every vertex $z \in (U_n \cup V_n) \setminus \{x_{i+1}, y_{i+1}\}$. With \eqref{bga::eq:partition_relation} it follows
    \begin{align*}
        p_{x_i} p_{x_{i+1}} p_{x_{i+2}} 
            &= p_{x_i} \left( \sum_{z \in U_n \cup V_n} p_z - p_{y_{i+1}} \right) p_{x_{i+2}} \\
            &= p_{x_i} p_{x_{i+2}} - p_{x_i} p_{y_{i+1}} p_{x_{i+2}}.
    \end{align*}
    The first term in the sum on the right-hand side vanishes by \eqref{bga::eq:partition_relation}. Hence, we obtain
    \begin{align*}
        p_{x_i} p_{x_{i+1}} p_{x_{i+2}} = - p_{x_i} p_{y_{i+1}} p_{x_{i+2}} = - p_{y_i} p_{y_{i+1}} p_{y_{i+2}},
    \end{align*}
    and the statement follows immediately.
\end{proof}

\begin{lemma}
    \label{sub::lemma:loop_elements_commute}
    Let $\mu = [x; \, k_1 \dots k_m]$ and $\nu = [x; \, \ell_1 \dots \ell_{m^\prime}]$ be two loops in $Q_n$ starting at $x \in U_n \cup V_n$. Then $p_\mu p_\nu = p_\nu p_\mu$.
\end{lemma}

\begin{proof}
    Evidently, we have
    \begin{align*}
        p_\mu p_\nu &= p_{[x; \, k_1 \dots k_m \ell_1 \dots \ell_{m^\prime}]}
    \end{align*}
    as well as
    \begin{align*}
        p_\nu p_\mu &= p_{[x; \, \ell_1 \dots \ell_{m^\prime} k_1 \dots k_m]}.
    \end{align*}
    From the previous lemma we obtain immediately
    \begin{align*}
        p_\mu p_\nu = \pm p_\nu p_\mu.
    \end{align*}
    It only remains to check that the sign on the right-hand side of this identity is positive. To see this, note that the path $[x; \, k_1 \dots k_m]$ is a loop. Therefore, every $\ell < n$ features an even number of times in the sequence $k_1 \dots k_m$. Now, to move the index $\ell_1$ from the right to the left in the combined sequence $k_1 \dots k_m \ell_1 \dots \ell_{m^\prime}$, one needs to swap it first with $k_m$, then with $k_{m-1}$, and so on. Each time $\ell_1$ is swapped with an index $k_i \neq \ell_1$ the sign of the corresponding algebra element $p_{[x; \, \cdot]}$ changes. However, since there is an even number of indices $k_i \neq \ell_1$ this sign change occurs an even number of times. Hence, in the end we have
    \begin{align*}
        p_{[x; \, k_1 \dots k_m \ell_1 \dots \ell_{m^\prime}]} = p_{[x; \, \ell_1 k_1 \dots k_m \ell_2 \dots \ell_{m^\prime}]}.
    \end{align*}
    Repeating this argument successively for the indices $\ell_2, \dots, \ell_{m^\prime}$ yields the claim.
\end{proof}

\begin{thm} 
    \label{hyp::thm:subhomogeneous:hypercube_algebra_is_subhomogeneous}
    The $C^\ast$-algebra $C^\ast(Q_n)$ is $2^{n-1}$-subhomogeneous, and one has for all vertices $x \in U_n \cup V_n$
    \begin{align}
        \label{sub::eq:rank_one_representations}
        \mathrm{dim}(\rho(p_x) \mathcal{H}) \leq 1
    \end{align}
    for every irreducible representation $\rho$ of $C^\ast(Q_n)$ on a Hilbert space $\mathcal{H}$. We call such a representation \emph{rank-one}. 
\end{thm}

\begin{proof}
    Let $x \in U_n \cup V_n$. By Proposition \ref{prop::bipartite_graph_algebras:definition:dense_subset} the corner $p_x C^\ast(Q_n) p_x$ is the closed span of all elements $p_\mu$, where $\mu$ is a loop in $Q_n$ that starts at $x$. By Lemma \ref{sub::lemma:loop_elements_commute} the elements $p_\mu$ commute pairwise. Consequently, the corner $p_x C^\ast(Q_n) p_x$ is a commutative C*-algebra. From this, Equation \eqref{sub::eq:rank_one_representations} follows immediately with general C*-algebra arguments, see e.g. \cite[IV.1.1.7]{blackadar_operator_2006}. Evidently, then $C^\ast(Q_n)$ is $2^{n-1}$-subhomogeneous, for $1 = \sum_{x \in U_n} p_x$ and $|U_n| = 2^{n-1}$.
\end{proof}

\begin{corollary}
    The C*-algebra $C^\ast(Q_n)$ is nuclear for all $n \in \N \setminus \{0\}$.
\end{corollary}

\begin{remark}
    In view of Problem \ref{intro::problem:nuclearity_implied_by_K23_not_subgraph} we would like to show more generally for any graph $G$ with $K_{2,3} \not \subset G$ that the corners $p_x C^\ast(G) p_x$ are commutative.
    However, the arguments used in the proof of the previous theorem and its preceding lemmas rely heavily on the particular structure of $Q_n$ as a hypercube, and they cannot be easily generalized.
\end{remark}

\section{Rank-one representations and edge weightings}
\label{sec::hypercube_rank_one_representations}

In this section we describe the rank-one representations of $C^\ast(Q_n)$ using certain complex edge weightings of $Q_n$. This will provide the necessary tools for classifying all irreducible representations of $C^\ast(Q_n)$ in the next section.

Recall that we call a representation $\rho$ of $C^\ast(Q_n)$ rank-one if every projection $\rho(p_x)$ for $x \in U_n \cup V_n$ is either zero or a rank-one projection. In view of Theorem \ref{hyp::thm:subhomogeneous:hypercube_algebra_is_subhomogeneous}, every irreducible representation of $C^\ast(Q_n)$ is rank-one. 

\begin{definition}
    \label{reps::def:admissible_weightings}
    Let $c: E_n \to \C$ be a complex weighting of the edges of $Q_n$ and let 
    $$
    Q_n(c) = (U_n(c), V_n(c), E_n(c))
    $$
    be the subgraph of $Q_n$ induced by the edges $ij \in E_n$ with $c(ij) \neq 0$.
    We call the weighting $c$ \emph{admissible} if it satisfies
    \begin{align}
        \label{wei::eq:admissible_weighting_condition1}
        \sum_{i \in \mathcal{N}(j_1) \cap \mathcal{N}(j_2)} c(ij_1) \overline{c(ij_2)} = \delta_{j_1j_2} \qquad \text{ for all } j_1, j_2 \in V_n(c),
    \end{align}
    as well as
    \begin{align}
        \label{wei::eq:admissible_weighting_condition2}
        \sum_{j \in \mathcal{N}(i_1) \cap \mathcal{N}(i_2)} c(i_1j) \overline{c(i_2j)} = \delta_{i_1i_2} \qquad \text{ for all } i_1, i_2 \in U_n(c),
    \end{align}
    where $\mathcal{N}(x)$ with $x \in U_n \cup V_n$ denotes the set of neighbors of the vertex $x$ in $Q_n$.
\end{definition}

The following lemma is needed to construct rank-one representations from admissible edge weightings. It will also be crucial later when we exploit the description of rank-one representations by edge weightings to classify all irreducible representations of $C^\ast(Q_n)$.

\begin{lemma}
    \label{reps::lemma:admissible_weightings_and_squares}
    Let $c: E_n \to \C$ be an admissible edge weighting of the hypercube $Q_n$, i.e. $c$ satisfies \eqref{wei::eq:admissible_weighting_condition1} and \eqref{wei::eq:admissible_weighting_condition2}. Further, assume that $x_1, x_2, y_1, y_2 \in U_n \cup V_n$ are four vertices which induce a subgraph of $Q_n$ that is isomorphic to a square as in Figure \ref{reps::fig:square_with_edges_efgh}. 
    \begin{enumerate}
        \item If $c(e) \neq 0$ and $c(f) \neq 0$, then we have $c(g) \neq 0$ and $c(h) \neq 0$ as well, and
        \begin{align*}
            \frac{c(e)}{c(f)} = - \overline{\left( \frac{c(g)}{c(h)} \right)}.
        \end{align*}
        \item In that case, we also have 
        \begin{align*}
            |c(e)| = |c(g)| \quad \text{ and } \quad |c(f)| = |c(h)|.
        \end{align*}
        \item The induced subgraph $Q_n(c) \subset Q_n$ is a disjoint union of sub-hypercubes.
    \end{enumerate}
\end{lemma}

    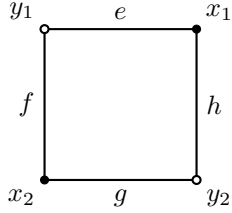
\begin{figure}[ht]
        \begin{tikzpicture}
            \coordinate (A) at (0,0);
            \coordinate (B) at (2,0);
            \coordinate (C) at (2,2);
            \coordinate (D) at (0,2);

            \draw[thick] (A) -- (B) node[midway, below] {$g$};
            \draw[thick] (B) -- (C) node[midway, right] {$h$};
            \draw[thick] (C) -- (D) node[midway, above] {$e$};
            \draw[thick] (D) -- (A) node[midway, left] {$f$};

            \filldraw[black] (A) circle (1.5pt) node[below left] {$x_2$};
            \filldraw[fill=white, draw=black, thick] (B) circle (1.5pt) node[below right] {$y_2$};
            \filldraw[black] (C) circle (1.5pt) node[above right] {$x_1$};
            \filldraw[fill=white, draw=black, thick] (D) circle (1.5pt) node[above left] {$y_1$};
        \end{tikzpicture}
        \label{reps::fig:square_with_edges_efgh}
        \caption{A square in the hypercube $Q_n$ with edges $e, f, g, h$}
    \end{figure}

\begin{proof}
    Ad (1). The vertices $y_1$ and $y_2$ have exactly two common neighbors in $Q_n$, namely $x_1$ and $x_2$, see Proposition \ref{hyp::prop:common_neighbors_in_hypercubes}. Therefore, condition \eqref{wei::eq:admissible_weighting_condition1} or \eqref{wei::eq:admissible_weighting_condition2} for weightings gives
    \begin{align*}
        c(e) \overline{c(h)} + c(f) \overline{c(g)} = 0.
    \end{align*}
    The statement follows immediately by a standard calculation.

    Ad (2). The previous statement entails
    \begin{align*}
        \left| \frac{c(e)}{c(f)} \right| = \left| \frac{c(g)}{c(h)} \right|,
    \end{align*}
    where all terms $c(e), c(f), c(g), c(h)$ are non-zero by assumption and (1). Similarly, one obtains
    \begin{align*}
        \left| \frac{c(f)}{c(g)} \right| = \left| \frac{c(h)}{c(e)} \right|.
    \end{align*}
    It follows
    \begin{align*}
        |c(e)| 
            = |c(g)| \left| \frac{c(f)}{c(g)} \right| \cdot \left| \frac{c(e)}{c(f)} \right|
            = |c(g)| \left| \frac{c(h)}{c(e)} \right| \cdot \left| \frac{c(g)}{c(h)} \right|
            = \frac{|c(g)|^2}{|c(e)|}.
    \end{align*}
    Therefore, $|c(e)|^2 = |c(g)|^2$, and the claim follows.

    Ad (3). In view of (1), $Q_n(c)$ has the following property: Whenever it contains two adjacent edges of a square in $Q_n$, then it contains all four edges of the square. Thus, the claim follows directly from Lemma \ref{hyp::lemma:properties_of_hypercubes2}(3).
\end{proof}

\begin{definition}
    \label{reps::def:representation_rho_c_induced_by_weighting}
    Given an admissible edge weighting $c: E_n \to \C$ of $Q_n$, let 
    $$
    \rho_c: C^\ast(Q_n) \to M_{U_n(c)} \cong B(\C^{U_n(c)})
    $$
    be the unique representation that extends the assignment
    \begin{align*}
        p_i &\mapsto E_{ii} \in M_{U_n(c)}, && i \in U_n(c), \\
        p_j &\mapsto \left[ c(i_1j) \overline{c(i_2j)} \right]_{i_1, i_2 \in U_n} \in M_{U_n(c)}, && j \in V_n(c), \\
        p_x &\mapsto 0, && x \in (U_n \cup V_n) \setminus (U_n(c) \cup V_n(c)).
    \end{align*}
    We say that the representation $\rho_c$ is induced by the weighting $c$. 
\end{definition}

\begin{proposition}
    \label{reps::prop:admissible_weighting_induces_representation}
    The representation $\rho_c$ in the previous definition is well-defined.
    In particular, for every $j \in V_n(c)$ the matrix $\rho_c(p_j)$ is the projection onto the span of the vector
    \begin{align*}
        \psi_j = \left[ c(ij) \right]_{i \in U_n(c)} \in \C^{U_n(c)}.
    \end{align*}
\end{proposition}

\begin{proof}
    The second claim is true since the vector $\psi_j$ is normalized due to \eqref{wei::eq:admissible_weighting_condition1}. Thus, the elements $\rho_c(p_x)$ are projections for all $x \in U_n \cup V_n$.
    We need to check that the operators $\rho_c(p_x)$ for $x \in U_n(c) \cup V_n(c)$ satisfy the relations \eqref{bga::eq:partition_relation} and \eqref{bga::eq:orthogonality_relation} from Definition \ref{bipartite_graph_algebras:definition}. First, let us note that $|U_n(c)| = |V_n(c)|$ holds since by Lemma \ref{reps::lemma:admissible_weightings_and_squares}(3) the graph $Q_n(c)$ is a disjoint union of sub-hypercubes.

    Evidently, the $E_{ii}$ form a partition of unity. At the same time \eqref{wei::eq:admissible_weighting_condition1} guarantees that the vectors $\psi_j$ are orthogonal to each other.  What's more, because of $|V_n(c)| = |U_n(c)|$ the vectors $\psi_j$ form an orthogonal basis of $\C^{U_n(c)}$. Hence, the projections onto the spans of the $\psi_j$ form a partition of unity in $B(\C^{U_n(c)})$ as well. Further, one readily checks that $\psi_j$ is orthogonal to $e_i \in \C^{U_n(c)}$ if $ij \not \in E_n$ since the $i$-th entry of $\psi_j$ vanishes. Thus, the universal property of $C^\ast(Q_n)$ yields the desired representation $\rho_c$.
\end{proof}

Let us shortly summarize the notation that was introduced so far. Given an (admissible) edge weighting $c: E_n \to \C$ of $Q_n$, we denote
\begin{itemize}
    \item by $Q_n(c) = (U_n(c), V_n(c), E_n(c))$ the subgraph of $Q_n$ induced by the edges where $c$ does not vanish,
    \item by $\rho_c: C^\ast(Q_n) \to B(\C^{U_n(c)})$ the representation of $C^\ast(Q_n)$ induced by $c$.
\end{itemize}

\begin{example}
    \label{reps::ex:admissible_weighting_example}
    As an example, consider the square $Q_2$ with edges $e, f, g, h$ as in Figure \ref{reps::fig:square_with_edges_efgh}. For any $t \in [0,1]$ and $\bar{t} := 1-t$ an admissible edge weighting is given by
    \begin{align*}
        c(e) = \sqrt{t}, \quad c(f) = \sqrt{\bar{t}}, \quad c(g) = \sqrt{t}, \quad c(h) = -\sqrt{\bar{t}}.
    \end{align*}
    We sketch this weighting in Figure \ref{reps::fig:admissible_weighting_example}.
    \begin{figure}[ht]
        \begin{tikzpicture}
            \coordinate (A) at (0,0);
            \coordinate (B) at (2,0);
            \coordinate (C) at (2,2);
            \coordinate (D) at (0,2);
            \draw[thick] (A) -- (B) node[midway, below] {$\sqrt{t}$};
            \draw[thick] (B) -- (C) node[midway, right] {$-\sqrt{\bar{t}}$};
            \draw[thick] (C) -- (D) node[midway, above] {$\sqrt{t}$};
            \draw[thick] (D) -- (A) node[midway, left] {$\sqrt{\bar{t}}$};
            \filldraw[black] (A) circle (1.5pt) node[below left] {$x_2$};
            \filldraw[fill=white, draw=black, thick] (B) circle (1.5pt) node[below right] {$y_2$};
            \filldraw[black] (C) circle (1.5pt) node[above right] {$x_1$};
            \filldraw[fill=white, draw=black, thick] (D) circle (1.5pt) node[above left] {$y_1$};
        \end{tikzpicture}
        \caption{An admissible edge weighting of $Q_2$}
        \label{reps::fig:admissible_weighting_example}
    \end{figure}
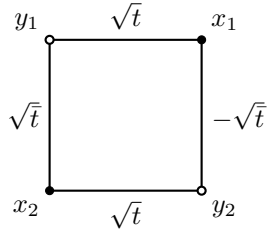
    If $t \in (0,1)$, then $Q_2(c) = Q_2$, but if $t =0$ ($t=1$), then $Q_2(c)$ is obtained from $Q_2$ by removing the edges $e,g$ ($f,h$, resp.). In any event, $Q_2(c) \subset Q_2$ is a disjoint union of sub-hypercubes. 
    
    To compute the representation $\rho_c$ let us assume that $\{x_1, x_2\} = U_2$ and $\{y_1, y_2\} = V_2$. Then, for any $t \in [0,1]$ the representation $\rho_c: C^\ast(Q_2) \to B(\C^2)$ is given
    \begin{align*}
        \rho_c(p_{x_1}) &= \begin{pmatrix}
            1 & 0 \\
            0 & 0
        \end{pmatrix}, & &&
        \rho_c(p_{y_1}) &= \begin{pmatrix}
            t & \sqrt{t \bar{t}} \\
            \sqrt{t \bar{t}} & \bar{t}
        \end{pmatrix} =
        \begin{pmatrix}
            \sqrt{t} & \sqrt{\bar{t}}
        \end{pmatrix}
        \begin{pmatrix}
            \sqrt{t} \\
            \sqrt{\bar{t}}
        \end{pmatrix}, \\
        \rho_c(p_{x_2}) &= \begin{pmatrix}
            0 & 0 \\
            0 & 1
        \end{pmatrix}, & &&
        \rho_c(p_{y_2}) &= \begin{pmatrix}
            \bar{t} & -\sqrt{t \bar{t}} \\
            -\sqrt{t \bar{t}} & t
        \end{pmatrix} =
        \begin{pmatrix}
            -\sqrt{\bar{t}} \\
            \sqrt{t}
        \end{pmatrix}
        \begin{pmatrix}
            -\sqrt{\bar{t}} &  \sqrt{t}
        \end{pmatrix}.
    \end{align*}
\end{example}

\begin{lemma}
    \label{hyp::lemma:properties_of_admissible_weightings}
    Let $c: E_n \to \C$ be an admissible weighting of the hypercube $Q_n$, and let $Q_n(c) \subset Q_n$ be as in Definition \ref{reps::def:admissible_weightings}. Then the following holds:
    \begin{enumerate}
        \item Assume that $Q_n(c) = \bigcup_\ell Q_n^{(\ell)}$ for disjoint sub-hypercubes $Q_n^{(\ell)} \subset Q_n$. Then the representation $\rho_c: C^\ast(Q_n) \to B(\C^{U_n(c)})$ induced by $c$ is a direct sum of irreducible representations $\rho^{(\ell)}$ which are vanishing on $p_x$ for $x \not \in Q_n^{(\ell)}$, respectively. In particular, $\rho_c$ is irreducible if and only if $Q_n(c)$ is a single sub-hypercube of $Q_n$.
        \item If $c^\prime: E_n \to \C$ is another admissible weighting of $Q_n$, then the representations $\rho_c$ and $\rho_{c^\prime}$ are unitarily equivalent if and only if $Q_n(c^\prime) = Q_n(c)$ and there are numbers $\lambda_x \in S^1$ for $x \in U_n(c) \cup V_n(c)$ such that
        \begin{align*}
            c^\prime(ij) = \lambda_i \lambda_j c(ij) \qquad \text{ for all } ij \in E_n(c).
        \end{align*}
        \item For any rank-one representation $\rho: C^\ast(Q_n) \to B(\mathcal{H})$ of the hypercube algebra $C^\ast(Q_n)$ on a Hilbert space $\mathcal{H}$ there exists an admissible weighting $c: E_n \to \C$ such that $\rho$ is unitarily equivalent to the representation $\rho_c$.
    \end{enumerate}
\end{lemma}

\begin{proof}
    Ad (1). First, let $U_n^{(\ell)} := U_n \cap Q_n^{(\ell)} \subset U_n(c)$ be the set of vertices from $U_n(c)$ contained in $Q_n^{(\ell)}$ for all $\ell$. Then the partition
    \begin{align*}
        U_n(c) = \bigcup_\ell U_n^{(\ell)}
    \end{align*}
    yields a natural block matrix form for the matrices in $M_{U_n(c)} = B(\C^{U_n(c)})$. In particular, the diagonal blocks correspond to the sub-hypercubes $Q_n^{(\ell)}$. We claim that $\rho_c(p_j)$ vanishes outside the diagonal block corresponding to the sub-hypercube that contains $j$. 
    
    Indeed, one readily checks that $\rho_c(p_j)$ has a non-zero entry in the $(i, i^\prime)$-th position if and only if $ij \in E_n(c)$ and $i^\prime j \in E_n(c)$. This is the case if and only if both $i$ and $i^\prime$ are contained in the same sub-hypercube $Q_n^{(\ell)}$ as the vertex $j$. Consequently, $\rho_c(p_j)$ is non-vanishing only on the diagonal block corresponding to the sub-hypercube $Q_n^{(\ell)}$ that contains $j$. It follows that $\rho_c$ is a direct sum of representations $\rho^{(\ell)}$ which are vanishing on $p_x$ for $x \not \in Q_n^{(\ell)}$, respectively.
    
    Further, if $Q_n(c)$ consists of a single sub-hypercube, then for any two vertices $i, i^\prime \in U_n(c)$ there is a path $x_1 x_2 \dots x_m$ in $Q_n(c)$ which connects $i$ and $i^\prime$. One checks that the matrix
    \begin{align*}
        \rho(p_{x_1}) \rho(p_{x_2}) \cdots \rho(p_{x_m})
    \end{align*}
    has a single non-zero entry in the $(i, i^\prime)$-th position. Hence, $\rho_c(C^\ast(Q_n)) = M_{U_n(c)}$ and the representation is irreducible.
    
    Ad (2). If the representations $\rho_c$ and $\rho_{c^\prime}$ are unitarily equivalent, then the sets of vertices $x \in U_n \cup V_n$ with $\rho_c(p_x) \neq 0$ and $\rho_{c^\prime}(p_x) \neq 0$, respectively, coincide. Consequently, $U_n(c) = U_n(c^\prime)$. Let $\mathcal{U} \in B(\C^{U_n(c)})$ be a unitary matrix such that
    \begin{align*}
        \rho_{c^\prime}(\cdot) = \mathcal{U} \rho_c(\cdot) \mathcal{U}^\ast
    \end{align*}
    holds. Because of $\rho_{c^\prime}(p_i) = E_{ii} = \rho_c(p_i)$ for all $i \in U_n(c)$, the unitary $\mathcal{U}$ must be diagonal, and there are numbers $\lambda_i \in S^1$ such that
    \begin{align*}
        \mathcal{U} = \mathrm{diag}((\lambda_i)_{i \in U_n(c)}).
    \end{align*}
    It is not hard to see that $\mathcal{U} \rho_c(p_j) \mathcal{U}^\ast$ is the projection onto the span of $\mathcal{U} \psi_j$ for all $j \in V_n(c)$ where $\psi_j = (c(ij))_{i \in U_n(c)}$ is the vector from Proposition \ref{reps::prop:admissible_weighting_induces_representation}. Let $\psi_j^\prime = (c^\prime(ij))_{i \in U_n(c)} \in \C^{U_n(c)}$ for $j \in V_n(c)$ be the vectors such that $\rho_{c^\prime}(p_j)$ is the projection onto the span of $\psi_j^\prime$. Then, $\mathcal{U} \rho_c(p_j) \mathcal{U}^\ast = \rho_{c^\prime}(p_j)$ implies that the vectors
    \begin{align*}
        \mathcal{U} \psi_j = (\lambda_i c(ij))_{i \in U_n(c)} \quad \text{ and } \quad \psi_j^\prime = (c^\prime(ij))_{i \in U_n(c)}
    \end{align*}
    have the same span. This is true if and only if there are numbers $\lambda_j \in \C \setminus \{0\}$ for $j \in V_n(c)$ with
    \begin{align*}
        \lambda_i \lambda_j c(ij) = c^\prime(ij) \qquad \text{ for all } ij \in E_n(c).
    \end{align*}
    Because of conditions \eqref{wei::eq:admissible_weighting_condition1} and \eqref{wei::eq:admissible_weighting_condition2} the numbers $\lambda_j$ must be of modulus one as well. 
    
    This proves that unitarily equivalent induced representations $\rho_c$ and $\rho_{c^\prime}$ must satisfy the condition from the statement. Conversely, it is easy to see that this condition implies unitary equivalence of $\rho_c$ and $\rho_{c^\prime}$.

    Ad (3). Assume that $\rho: C^\ast(Q_n) \to B(\mathcal{H})$ is a rank-one representation of the hypercube algebra $C^\ast(Q_n)$ on a Hilbert space $\mathcal{H}$. Let $Q_n(\rho) = (U_n(\rho), V_n(\rho), E_n(\rho)) \subset Q_n$ be the subgraph induced by the vertices $x \in U_n \cup V_n$ with $\rho(p_x) \neq 0$. By applying a unitary transformation if necessary we may assume $\mathcal{H} = \C^{U_n(\rho)}$ and $\rho(p_i) = E_{ii}$ for all $i \in U_n(\rho)$. For $j \in V_n(\rho)$, $\rho(p_j)$ is the projection onto the span of some normalized vector $\psi_j \in \C^{U_n(\rho)}$. Let a weighting $c: E_n \to \C$ be given by
    \begin{align*}
        c(ij) := \begin{cases}
            [\psi_j]_i,     & \text{ if } \rho(p_j) \neq 0, \\
            0,              & \text{ otherwise.}
        \end{cases} 
    \end{align*}
    Note that $[\psi_j]_i = 0$ if $ij \not \in E_n$ since in this case $\rho(p_j)$ and $\rho(p_i)$ are orthogonal.
    One has $|U_n(\rho)| = |V_n(\rho)|$ for both the $\rho(p_i)$ and $\rho(p_j)$ for $i \in U_n, j \in V_n$ form a partition of unity in $M_{U_n(\rho)}$. Hence, the $\psi_j$ form an orthonormal basis of $\C^{U_n(\rho)}$ and satisfy \eqref{wei::eq:admissible_weighting_condition1}. The matrix which contains the $\psi_j$ as columns is, thus, unitary, and it follows that its rows form an orthonormal basis of $\C^{U_n(\rho)}$ as well. Therefore, \eqref{wei::eq:admissible_weighting_condition2} is satisfied, and $c$ is an admissible weighting of $Q_n$.
\end{proof}

To summarize, every rank-one representation of $C^\ast(Q_n)$ is (up to unitary equivalence) induced by an admissible edge weighting of $Q_n$. However, two different edge weightings can define unitarily equivalent representations. Therefore, we introduce below a particular class of edge weightings $c_\mathbf{t}$ depending on points $\mathbf{t}$ in a standard simplex. In the next section we will prove that the induced representations $\rho_\mathbf{t}$ cover every irreducible representation of $C^\ast(Q_n)$ in a suitable sense.

The definition below is a straight generalization of the edge weightings of $Q_2$ from Example \ref{reps::ex:admissible_weighting_example}.

\begin{definition}
    \label{weig::def:representation_rho_t}
    For every tuple $\mathbf{t} = [t_0, \dots, t_{n-1}] \in \Delta_{n-1}$ let an edge weighting $c_\mathbf{t}$ of $Q_n$ be given by
    \begin{align*}
        c_\mathbf{t}(ij) = (-1)^{\mathrm{par}_k(i)} \sqrt{t_k} \qquad \text{ for all } ij \in E_n,
    \end{align*}
    where $k < n$ is the unique number with $j = i \# k$ and $\mathrm{par}_k(i)$ is the sum of the first $(k+1)$-binary digits of $i$ modulo $2$ as in Notation \ref{pre::notation:binary_numbers}. 
    Let 
    $$
        \rho_\mathbf{t}: C^\ast(Q_n) \to M_{U_n}
    $$    
    be the representation of $C^\ast(Q_n)$ induced by the weighting $c_\mathbf{t}$.
\end{definition} 

\begin{proposition}
    The representation $\rho_\mathbf{t}$ in the previous definition is well-defined, i.e. the edge weighting $c_\mathbf{t}$ is admissible for all $\mathbf{t} \in \Delta_{n-1}$ and its induced subgraph $Q_n(c_\mathbf{t})$ satisfies $U_n(c_\mathbf{t}) = U_n$.
\end{proposition}

\begin{proof}
    Let us first check $U_n(c_\mathbf{t}) = U_n$. Since $\mathbf{t} = [t_0, \dots, t_{n-1}] \in \Delta_{n-1}$ satisfies $\sum_k t_k = 1$, there exists at least one index $\ell$ with $t_\ell \neq 0$. Consequently, for every $i \in U_n$ one has $c_\mathbf{t}(ij) \neq 0$ if one chooses $j := i \# \ell$. Thus, $i$ is contained in the subgraph $Q_n(c_\mathbf{t})$, and it follows that $U_n(c_\mathbf{t}) = U_n$.

    It remains to check that $c_\mathbf{t}$ satisfies the conditions \eqref{wei::eq:admissible_weighting_condition1} and \eqref{wei::eq:admissible_weighting_condition2} from Definition \ref{reps::def:admissible_weightings}. One immediately observes for any fixed $j \in V_n$ and $i \in U_n$, respectively,
    \begin{align*}
        \sum_{i \in U_n} |c_\mathbf{t}(ij)|^2 = \sum_{k=0}^{n-1} t_k = 1 = \sum_{k=0}^{n-1} t_k = \sum_{j \in V_n} |c_\mathbf{t}(ij)|^2.
    \end{align*}
    On the other hand, for distinct $j_1, j_2 \in U_n$ recall from Proposition \ref{hyp::prop:common_neighbors_in_hypercubes} that there are either no or exactly two common neighbors $i_1, i_2 \in \mathcal{N}(j_1) \cap \mathcal{N}(j_2)$. In the latter case let $\ell < k$ be such that $j_1 = i_1 \# \ell$, $j_2 = i_2 \# \ell$ as well as $j_1 = i_2 \# k$ and $j_2 = i_1 \# k$ hold.
    Then it is
    \begin{align*}
        c(i_1 j_1) &\overline{c(i_1 j_2)} + c(i_2 j_1) \overline{c(i_2 j_2)} \\
            &= (-1)^{\mathrm{par}_\ell(i_1) + \mathrm{par}_k(i_1)} \sqrt{t_\ell} \sqrt{t_k} + (-1)^{\mathrm{par}_k(i_2) + \mathrm{par}_\ell(i_2)} \sqrt{t_k} \sqrt{t_\ell}.
    \end{align*}
    One checks that the two terms in this sum have opposite signs, and therefore the sum vanishes. Indeed, the binary representation of $i_1$ and $i_2$ differs exactly in the $\ell$-th and $k$-thm digit. As $\ell < k$ it follows $\mathrm{par}_\ell(i_1) \neq \mathrm{par}_\ell(i_2)$ and $\mathrm{par}_k(i_1) = \mathrm{par}_k(i_2)$.
    Thus, \eqref{wei::eq:admissible_weighting_condition1} is satisfied. The other condition \eqref{wei::eq:admissible_weighting_condition2} can be checked in the same way. 
\end{proof}

\begin{example}
    \label{hyp::example:representation_rho_t_for_the_cube_Q3}
    For $n=3$ the vertices in $U_3$ are labeled by the binary numbers $000, 011, 101$ and $110$. Similarly, the vertices in $V_3$ are labeled by $001, 010, 100, 111$. Thus, for $\mathbf{t} = [a^2, b^2, c^2] \in \Delta_2$ one verifies that $\rho_\mathbf{t}(p_j)$ for $j \in V_3$ is given by the projection onto $\psi_j \in \C^{U_3}$, respectively, where
    \begin{align*}
        \psi_{001} = \begin{pmatrix}
            a \\ b \\ c \\ 0
        \end{pmatrix},
        \quad
        \psi_{010} = \begin{pmatrix}
            b \\ - a \\ 0 \\ c
        \end{pmatrix},
        \quad
        \psi_{100} = \begin{pmatrix}
            c \\ 0 \\ - a \\ - b
        \end{pmatrix},
        \quad
        \psi_{111} = \begin{pmatrix}
            0 \\ c \\ - b \\ a
        \end{pmatrix}.
    \end{align*}
    Here, we identify $\C^{U_3}$ with $\C^4$ via
    \begin{align*}
        e_{000} := \begin{pmatrix}
            1 \\ 0 \\ 0 \\ 0
        \end{pmatrix},
        \quad
        e_{011} := \begin{pmatrix}
            0 \\ 1 \\ 0 \\ 0
        \end{pmatrix},
        \quad
        e_{101} := \begin{pmatrix}
            0 \\ 0 \\ 1 \\ 0
        \end{pmatrix},
        \quad
        e_{110} := \begin{pmatrix}
            0 \\ 0 \\ 0 \\ 1
        \end{pmatrix}.
    \end{align*}
\end{example}

At the end of this section we prove a technical lemma which describes $\rho_\mathbf{t}$ for a boundary point $\mathbf{t} \in \Delta_{n-1}$ as a direct sum of smaller representations $\rho_\mathbf{s}$ for suitable $\mathbf{s} \in \Delta_{n-2}$. This will be useful in the proof of Theorem \ref{irreps::thm:classification_of_irreps} where the irreducible representations of $C^\ast(Q_n)$ are classified.
To state the lemma let us introduce some notation.

\begin{notation}
    \label{weig::notation:pi_plus_minus}
    For $\ell < n$ and $y = \sum_{k=0}^{n-2} y_k 2^k < 2^{n-1}$ with $y_k \in \{0,1\}$, let
    \begin{align*}
        y^{(+\ell)} &:= \sum_{k=0}^{\ell-1} y_k 2^k + 2^\ell + \sum_{k=\ell}^{n-2} y_k 2^{k+1} < 2^n, \\
        y^{(-\ell)} &:= \sum_{k=0}^{\ell-1} y_k 2^k + \sum_{k=\ell}^{n-2} y_k 2^{k+1} < 2^n,
    \end{align*}
    be the numbers obtained from $y$ by inserting a $1$ or a $0$, respectively, at the $\ell$-th position in the binary representation of $y$. 

    Further, let $\pi^{(+\ell)}, \pi^{(-\ell)}: C^\ast(Q_n) \to C^\ast(Q_{n-1})$ be surjective $*$-homomorphisms given by
    \begin{align*}
        \pi^{(+\ell)}: p_x &\mapsto \begin{cases}
            \tilde{p}_y, & \exists y \in U_{n-1} \cup V_{n-1}: x = (y \# 0)^{(+\ell)}, \\
            0,           & \text{ otherwise},
        \end{cases} \\
        \pi^{(-\ell)}: p_x &\mapsto \begin{cases}
            \tilde{p}_y, & \exists y \in U_{n-1} \cup V_{n-1}: x = y^{(-\ell)}, \\
            0,           & \text{ otherwise}.
        \end{cases}
    \end{align*}
    Here we write $\tilde{p}_y$ for the canonical generators of $C^\ast(Q_{n-1})$ to distinguish them from the generators $p_x$ of $C^\ast(Q_n)$.
\end{notation}

Note that one has 
$$
U_n = \{y^{(-\ell)} \mid y \in U_{n-1}\} \cup \{(y \# 0)^{(+\ell)} \mid y \in U_{n-1}\}
$$ 
for all $\ell < n$. In fact, a number $x < 2^n$ is in $U_n$ iff its binary representation contains an even number of $1$s. 
If $x < 2^n$ is in $U_n$, then it has either a $0$ in the $\ell$-th position (in which case $x = y^{(-\ell)}$ for a suitable $y \in U_{n-1}$) or it has a $1$ in the $\ell$-th position (in which case $x = (y \# 0)^{(+\ell)}$ for a suitable $y \in U_{n-1}$).

The $\ast$-homomorphisms $\pi^{(+\ell)}$ and $\pi^{(-\ell)}$ exist by the universal property of $C^\ast(Q_n)$.

\begin{lemma}
    \label{reps::lemma:decomposition_of_rho_t_for_boundary_points}
    For $\mathbf{s} = [s_0, \dots, s_{n-2}] \in \Delta_{n-2}$ and $\ell < n$ let $\mathbf{t} \in \Delta_{n-1}$ be given by
    \begin{align*}
        \mathbf{t} = [t_0, \dots, t_{n-1}] := [s_0, \dots, s_{\ell-1}, 0, s_\ell, \dots, s_{n-2}].
    \end{align*}
    Then one has the unitary equivalence
    \begin{align*}
        \rho_\mathbf{t} \cong \left(\rho_\mathbf{s} \circ \pi^{(+\ell)}\right) \oplus \left(\rho_\mathbf{s} \circ \pi^{(-\ell)}\right).
    \end{align*}
\end{lemma}

\begin{proof}
    The representation $\rho_\mathbf{t}$ acts on $\C^{U_n}$, while the direct sum on the right-hand side of the equation in the statement acts on $\C^{U_{n-1}} \oplus \C^{U_{n-1}}$. Consider the unitary operator $\mathcal{U}: \C^{U_{n-1}} \oplus \C^{U_{n-1}} \to \C^{U_n}$ given by
    \begin{align*}
        \mathcal{U}: (e_x, 0) \mapsto e_{(x \# 0)^{(+\ell)}}, \qquad
        (0, e_x) \mapsto e_{x^{(-\ell)}},
    \end{align*}
    for all $x \in U_{n-1}$. 
    Let us determine the representation
    \begin{align*}
        \sigma(\cdot) := \mathcal{U} \left(\left(\rho_\mathbf{s} \circ \pi^{(+\ell)}(\cdot)\right) \oplus \left(\rho_\mathbf{s} \circ \pi^{(-\ell)}(\cdot)\right) \right) \mathcal{U}^\ast
    \end{align*}
    of $C^\ast(Q_n)$ on $\C^{U_n}$. For $x \in U_n$ with $[x]_\ell = 1$ we have $x=(y \# 0)^{(+\ell)}$ for a suitable (and unique) $y \in U_{n-1}$. Then, $\pi^{(+\ell)}(p_x) = \tilde{p}_y$ and $\pi^{(-\ell)}(p_x) = 0$. Thus, one easily checks
    \begin{align*}
        \sigma(p_x) = \mathcal{U} \begin{pmatrix}
            \tilde{E}_{yy} & 0 \\
            0      & 0
        \end{pmatrix} \mathcal{U}^\ast = E_{xx},
    \end{align*}
    where we write a bounded operator on $\C^{U_{n-1}} \oplus \C^{U_{n-1}}$ as a $2 \times 2$ block matrix in a natural way, and $\tilde{E}_{yy} \in M_{U_{n-1}}$ is the usual standard matrix unit. One obtains 
    \begin{align*}
        \sigma(p_x) = E_{xx}
    \end{align*}
    for $x \in U_n$ with $[x]_\ell = 0$ similarly.

    Next, let $x \in V_n$ with $[x]_\ell = 1$, and let $y \in V_{n-1}$ be the unique vertex with $x = (y \# 0)^{(+\ell)}$. Then $\pi^{(+\ell)}(p_x) = \tilde{p}_y$ and $\pi^{(-\ell)}(p_x) = 0$. Further, $\tilde{p}_y$ is the projection onto the span of the vector $\tilde{\psi}_y \in \C^{U_{n-1}}$ with entries
    \begin{align*}
        [\tilde{\psi}_y]_i = \begin{cases}
            (-1)^{\mathrm{par}_k(i)} \sqrt{s_k}     &\exists \, k < n-1 : y = i \# k, \\
            0                                       &\text{ otherwise},
        \end{cases}
    \end{align*}
    where $i$ ranges over $U_{n-1}$.
    Consequently, $\sigma(p_x)$ is the projection onto the span of the vector $\varphi_x := \mathcal{U} (\tilde{\psi}_y, 0) \in \C^{U_n}$. Since, the second component of $(\tilde{\psi}_y, 0)$ vanishes it is clear that $[\varphi_x]_i$ vanishes for $i \in U_n$ with $[i]_\ell = 0$. For the other indices $i \in U_n$ there exists a unique $\tilde{i} \in U_{n-1}$ with $i = (\tilde{i} \# 0)^{(+\ell)}$, and one computes
    \begin{align*}
        [\varphi_x]_{i} &= \begin{cases}
            (-1)^{\mathrm{par}_k(\tilde{i})} \sqrt{s_k}     &\exists \, k < n-1 : y = \tilde{i} \# k, \\
            0                                       &\text{ otherwise.}
        \end{cases}
    \end{align*}
    One checks that $y = \tilde{i} \# k$ is equivalent to
    \begin{align*}
            x &= (y \# 0)^{(+\ell)} = (\tilde{i} \# 0)^{(+\ell)} \# k = i \# k, & \text{ if } k < \ell, \\
            x &= (y \# 0)^{(+\ell)} = (\tilde{i} \# 0)^{(+\ell)} \# (k+1) = i \# (k+1), & \text{ if } k \geq \ell.
    \end{align*}
    Further,
    \begin{align*}
        \mathrm{par}_k(\tilde{i}) = \begin{cases}
            - \mathrm{par}_k((\tilde{i} \# 0)^{(+\ell)}) = -\mathrm{par}_k(i)    &\text{ if } k < \ell, \\
            \mathrm{par}_{k+1}((\tilde{i} \# 0)^{(+\ell)}) = \mathrm{par}_{k+1}(i)   &\text{ if } k \geq \ell,
        \end{cases}
    \end{align*}
    Using $\mathbf{t} = [s_0, \dots, s_{\ell-1}, 0, s_\ell, \dots, s_{n-2}]$ we thus obtain
    \begin{align*}
        [\varphi_x]_{i} &= \begin{cases}
            -(-1)^{\mathrm{par}_k(i)} \sqrt{t_k}        &\exists \, k < \ell : x = i \# k, \\
            (-1)^{\mathrm{par}_{k+1}(i)} \sqrt{t_{k+1}}       &\exists \, k \geq \ell : x = i \# (k+1), \\
            0                                           &\text{ otherwise.}
        \end{cases}
    \end{align*}
    Analogously, one checks for $x \in V_n$ with $[x]_\ell = 0$ that $\sigma(p_x)$ is the projection onto the span of the vector $\varphi_x \in \C^{U_n}$ with
    \begin{align*}
        [\varphi_x]_i = \begin{cases}
            (-1)^{\mathrm{par}_k(i)} \sqrt{t_k}                &\exists \, k < \ell : x = i \# k, \\
            (-1)^{\mathrm{par}_{k+1}(i)} \sqrt{t_{k+1}}         &\exists \, k \geq \ell : x = i \# (k+1), \\
            0                                                   &\text{ otherwise.}  \\                                
        \end{cases}
    \end{align*}
    
    Putting everything together and using $t_\ell = 0$, we obtain that $\sigma$ is induced by the edge weighting $c: E_n \to \C$ with
    \begin{align*}
        c(ij) = \begin{cases}
            \varepsilon_{j}^k (-1)^{\mathrm{par}_k(i)} \sqrt{t_k}        &\exists \, k < n : j = i \# k, \\
            0                                           &\text{ otherwise },
        \end{cases}
    \end{align*}
    for all edges $ij \in E_n$, where
    \begin{align*}
        \varepsilon_j^k = \begin{cases}
            -1,    & \text{ if } [j]_\ell = 1 \text{ and } k < \ell, \\
            1,     & \text{ otherwise.}
        \end{cases}
    \end{align*}
    Setting $\lambda_x := (-1)^{[x]_\ell \mathrm{par}_\ell(x)}$ for all $x \in U_n \cup V_n$ it is not hard to verify
    \begin{align*}
        \lambda_i \lambda_j c(ij) = c_\mathbf{t}(ij),
    \end{align*}
    where $c_\mathbf{t}$ is the edge weighting from Definition \ref{weig::def:representation_rho_t}. Thus, $\sigma$ is unitarily equivalent to $\rho_\mathbf{t}$ by Lemma \ref{hyp::lemma:properties_of_admissible_weightings}(2).
\end{proof}

\section{Irreducible representations of hypercube C*-algebras}
\label{sec::irreducible_representations}

Building on the previous sections, we show that any irreducible representation of the hypercube algebra $C^\ast(Q_n)$ is unitarily equivalent to a subrepresentation of the representation $\rho_\mathbf{t}$ from Definition \ref{weig::def:representation_rho_t} for some point $\mathbf{t} = [t_0, \dots, t_{n-1}] \in \Delta_{n-1}$. We call a representation $\rho_1$ a subrepresentation of a representation $\rho$ if there is a third representation $\rho_2$ such that $\rho = \rho_1 \oplus \rho_2$ holds.

Let us briefly outline the proof strategy. We know from the previous sections that every irreducible representation $\rho$ is unitarily equivalent to a representation $\rho_c$ induced by an admissible edge weighting $c: E_n \to \C$. Further, since $\rho_c$ is irreducible, the induced subgraph $Q_n(c) \subset Q_n$ is a single sub-hypercube of $Q_n$. It will turn out that it suffices to prove $\rho_c \cong \rho_\mathbf{t}$ for suitable $\mathbf{t} \in \Delta_{n-1}$ in the case where $Q_n(c) = Q_n$. In this case the edge weighting $c$ is nowhere-vanishing, and we prove the statement in two steps:

\begin{itemize}
    \item First, in Lemma \ref{hyp::lemma:nonvanishing_admissible_weightings_are_almost_ct} we show that any nowhere-vanishing admissible weighting $c$ satisfies $|c(ij)| = \sqrt{t_k}$ for all $ij \in E_n$ with $i = j \# k$ for a suitable point $\mathbf{t} = [t_0, \dots, t_{n-1}] \in \Delta_{n-1}$.
    \item Then, in Lemma \ref{hyp::lemma:nonvanishing_admissible_weightings_are_unitarily_equivalent_to_ct} we prove that such a representation $\rho_c$ is unitarily equivalent to $\rho_\mathbf{t}$. For that we use Lemma \ref{hyp::lemma:properties_of_admissible_weightings}(2) which says that $\rho_c \cong \rho_\mathbf{t}$ holds if and only if there are $\lambda_x \in S^1$ for $x \in U_n(c) \cup V_n(c)$ such that
    \begin{align*}
        c(ij) = \lambda_i \lambda_j c_\mathbf{t}(ij) \qquad \text{ for all } ij \in E_n(c).
    \end{align*}
\end{itemize}

\begin{lemma}
    \label{hyp::lemma:nonvanishing_admissible_weightings_are_almost_ct}
    Assume that $c: E_n \to \C$ is a nowhere-vanishing admissible edge weighting of the hypercube $Q_n$, i.e. $c(ij) \neq 0$ for all edges $ij \in E_n$. 
    
    There is a point $\mathbf{t} = [t_0, \dots, t_{n-1}] \in \Delta_{n-1}$ such that $|c(ij)| = \sqrt{t_k}$ holds for all edges $ij \in E_n$ where $i = j \# k$.
\end{lemma}

\begin{proof}
    Since $c$ is nowhere-vanishing, we have $Q_n(c) = Q_n$. Let us set
    \begin{align*}
        t_k = |c(0 [0 \# k])|^2
    \end{align*}
    for all $k < n$. Note hat $\sum_{k=0}^{n-1} t_k = 1$ holds by condition \eqref{wei::eq:admissible_weighting_condition2} for admissible edge weightings, so that $\mathbf{t} = [t_0, \dots, t_{n-1}]$ is indeed in $\Delta_{n-1}$.
    
    It remains to show $|c(ij)| = \sqrt{t_k}$ for all edges $ij \in E_n$ with $i = j \# k$. Recall that we denote the set of all these edges by $E_n^{(k)}$. Now, the claim follows directly from Lemma \ref{reps::lemma:admissible_weightings_and_squares}(2) in combination with Lemma \ref{hyp::lemma:properties_of_hypercubes2}(2), since the subset of edges in $E_n^{(k)}$ with $|c(e)| = \sqrt{t_k}$ is non-empty (it contains $0 [0 \# k]$) and closed under taking opposite edges in a square, i.e. under the equivalence relation from Lemma \ref{hyp::lemma:properties_of_hypercubes2}(2).
\end{proof}

\begin{lemma}
    \label{hyp::lemma:nonvanishing_admissible_weightings_are_unitarily_equivalent_to_ct}
    Let $c: E_n \to \C$ be a nowhere-vanishing admissible weighting of the hypercube $Q_n$. Then there is a tuple $\mathbf{t} = [t_0, \dots, t_{n-1}] \in \Delta_{n-1}$ such that $\rho_c$ is unitarily equivalent to $\rho_{\mathbf{t}}$ where the latter is the representation induced by the weighting $c_\mathbf{t}$ from Definition \ref{weig::def:representation_rho_t}. 
\end{lemma}

\begin{proof}
    In view of Lemma \ref{hyp::lemma:properties_of_admissible_weightings}(2) the statement is equivalent to the existence of numbers $\lambda_x \in S^1$ for $x \in U_n \cup V_n$ with
    \begin{align}
        \label{hyp::eq:c_unitarily_equivalent_to_ct}
        c(ij) = \lambda_i \lambda_j c_\mathbf{t}(ij) 
    \end{align}
    for all edges $ij \in E_n$.

    First, recall Lemma \ref{hyp::lemma:nonvanishing_admissible_weightings_are_almost_ct} and choose $\mathbf{t} = [t_0, \dots, t_{n-1}] \in \Delta_{n-1}$ such that 
    $$
    |c(ij)| = \sqrt{t_k} \quad \text{ for all } ij \in E_n \text{ where } i = j \# k
    $$
    holds.
    We prove \eqref{hyp::eq:c_unitarily_equivalent_to_ct} for all edges $ij \in E_n^{(0)} \cup E_n^{(1)} \cup \dots \cup E_n^{(K)}$ by induction over $K$, where $E_n^{(k)}$ is the set of edges $ij \in E_n$ with $i = j \# k$ and $k < n$. For $K=0$ one simply chooses 
    \begin{align*}
        \lambda_i := (-1)^{\mathrm{par}_0(i)} \frac{c(ij)}{|c(ij)|} \qquad \text{ for all } i \in U_n \text{ and } j := i \# 0
    \end{align*}
    as well as $\lambda_j := 1$ for all $j \in V_n$. Then \eqref{hyp::eq:c_unitarily_equivalent_to_ct} holds for all edges $ij \in E_n^{(0)}$ since 
    \begin{align*}
        \lambda_i \lambda_j c_\mathbf{t}(ij) = (-1)^{\mathrm{par}_0(i)}  \frac{c(ij)}{|c(ij)|} (-1)^{\mathrm{par}_0(i)} \sqrt{t_0} = c_{ij}
    \end{align*}
    is true for all edges $ij \in E_n^{(0)}$.

    For the induction step, we may
    assume without loss of generality
    \begin{align*}
        c_\mathbf{t}(ij) = c(ij) \qquad \text{ for all } ij \in E_n^{(\ell)} \text{ for some } \ell < K.
    \end{align*}
    Let us choose
    \begin{align*}
        \lambda_i := (-1)^{\mathrm{par}_K(i)} \frac{c(ij)}{|c(ij)|} \qquad \text{ for all } i \in U_n \text{ with } [i]_K = 0 \text{ and } j := i \# K,
    \end{align*}
    as well as
    \begin{align*}
        \lambda_j := (-1)^{\mathrm{par}_K(i)} \frac{c(ij)}{|c(ij)|} \qquad \text{ for all } j \in V_n \text{ with } [j]_K = 0 \text{ and } i := j \# K.
    \end{align*}
    For all other $i \in  U_n$, $j \in V_n$ simply set $\lambda_i = \lambda_j = 1$. We first check \eqref{hyp::eq:c_unitarily_equivalent_to_ct} for all edges $ij \in E_n^{(\ell)}$ with $\ell < K$. 
    
    Indeed, if $[i]_K = 1$, then $[j]_K = 1$ holds as well and the equality is evident. Otherwise, we have both $[i]_K = 0$ and $[j]_K = 0$, where $i$ and $j$ differ in the $\ell$-th binary digit. Therefore, one obtains
    \begin{align}
        \label{hyp::eq:induction_step_for_weightings_consistency_with_previous_levels}
        \begin{aligned}
        \lambda_i \lambda_j c_\mathbf{t}(ij) 
            &= (-1)^{\mathrm{par}_K(i)} \frac{c(i[i \# K])}{|c(i[i \# K])|} (-1)^{\mathrm{par}_K(j \#K)} \frac{c([j \# K] j)}{|c([j \# K] j)|} c(ij),
        \end{aligned}
    \end{align}
    where we use the induction hypothesis to get $c_\mathbf{t}(ij) = c(ij)$. The vertices $i, j, i \# K, j \# K$ induce a square in $Q_n$ as sketched below.
    
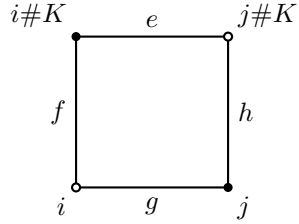
\begin{center}
        \begin{tikzpicture}
            \coordinate (A) at (0,0);
            \coordinate (B) at (2,0);
            \coordinate (C) at (2,2);
            \coordinate (D) at (0,2);

            \draw[thick] (A) -- (B) node[midway, below] {$g$};
            \draw[thick] (B) -- (C) node[midway, right] {$h$};
            \draw[thick] (C) -- (D) node[midway, above] {$e$};
            \draw[thick] (D) -- (A) node[midway, left] {$f$};

            \filldraw[fill=white, draw=black, thick] (A) circle (1.5pt) node[below left] {$i$};
            \filldraw[black] (B) circle (1.5pt) node[below right] {$j$};
            \filldraw[fill=white, draw=black, thick] (C) circle (1.5pt) node[above right] {$j\# K$};
            \filldraw[black] (D) circle (1.5pt) node[above left] {$i\# K$};
        \end{tikzpicture}
    \end{center}
    One immediately observes 
    \begin{align*}
       \mathrm{par}_K(i) = \mathrm{par}_K(j \# K) 
       \quad \text{ and } \quad
       \mathrm{par}_\ell(i) = \mathrm{par}_\ell(j \# K) + 1.
    \end{align*}
    From Lemma \ref{reps::lemma:admissible_weightings_and_squares}(1) and the induction hypothesis we get 
    \begin{align*}
        \frac{c(f)}{\overline{c(h)}} = - \frac{c(e)}{\overline{c(g)}} = - \frac{(-1)^{\mathrm{par}_\ell(j\# K)} \sqrt{t_\ell}}{(-1)^{\mathrm{par}_\ell(i)} \sqrt{t_\ell}} 
        = 1.
    \end{align*}
    Consequently, $c(i[i\# K]) = c(f) = \overline{c(h)} = \overline{c([j\# K] j)}$.
    Inserting this into \eqref{hyp::eq:induction_step_for_weightings_consistency_with_previous_levels} yields $\lambda_i \lambda_j c_\mathbf{t}(ij) = c(ij)$ as desired. 
    
    Thus, we have proved that the unitary transformation of the weighting $c$ given by the numbers $\lambda_i$ and $\lambda_j$ does not change the weighting of the edges $ij \in E_n^{(0)} \cup \dots \cup E_n^{(K-1)}$. To complete the induction step it only remains to show $c(ij) = \lambda_i \lambda_j c_\mathbf{t}(ij)$ for all $ij \in E_n^{(K)}$. For this, assume first $[i]_K = 0$. Then $[j]_K = 1$ and we have
    \begin{align*}
        \lambda_i \lambda_j c_\mathbf{t}(ij) = (-1)^{\mathrm{par}_K(i)} \frac{c(ij)}{|c(ij)|}  (-1)^{\mathrm{par}_K(i)} \sqrt{t_K} = c(ij).
    \end{align*}
    Otherwise, one has $[i]_K = 1$ and $[j]_K = 0$. It follows again
    \begin{align*}
        \lambda_i \lambda_j c_\mathbf{t}(ij) = (-1)^{\mathrm{par}_K(i)} \frac{c(ij)}{|c(ij)|}  (-1)^{\mathrm{par}_K(i)} \sqrt{t_K} = c(ij).
    \end{align*}
\end{proof}

\begin{thm}
    \label{irreps::thm:classification_of_irreps}
    Let $\rho: C^\ast(Q_n) \to B(\mathcal{H})$ be an irreducible representation of the hypercube algebra $C^\ast(Q_n)$ on a Hilbert space $\mathcal{H}$. Then there is a point $\mathbf{t} \in \Delta_{n-1}$ such that $\rho$ is unitarily equivalent to a subrepresentation of the representation $\rho_\mathbf{t}$ induced by the weighting $c_\mathbf{t}$ from Definition \ref{weig::def:representation_rho_t}.
\end{thm}

\begin{proof}
    We prove the statement by induction over $n$. For $n=1$ the hypercube $Q_1$ has two vertices and one edge, $C^\ast(Q_1) \cong \C$, $\Delta_0 = \{[1]\}$ and one readily checks that $\rho_{[1]}$ is the unique irreducible representation of $C^\ast(Q_1)$.

    For the induction step let $\rho: C^\ast(Q_n) \to B(\mathcal{H})$ be an irreducible representation. By Theorem \ref{hyp::thm:subhomogeneous:hypercube_algebra_is_subhomogeneous}, $\rho$ is rank-one, and in view of Lemma \ref{hyp::lemma:properties_of_admissible_weightings}(3) we may assume $\rho = \rho_c$ for some admissible weighting $c: E_n \to \C$ of the hypercube $Q_n$. Further, since $\rho$ is irreducible we know from the same lemma that $c$ is supported on a single sub-hypercube $Q_n(c)$ of $Q_n$.

    First assume $Q_n(c) = Q_n$. Then $c$ is nowhere-vanishing, and by Lemma \ref{hyp::lemma:nonvanishing_admissible_weightings_are_unitarily_equivalent_to_ct} there is a $\mathbf{t} \in \Delta_{n-1}$ such that $\rho_c$ is unitarily equivalent to $\rho_\mathbf{t}$.

    If $Q_n(c) \subsetneq Q_n$, then by Lemma \ref{hyp::lemma:properties_of_hypercubes2}(1) there is some $\ell < n$ such that $Q_n(c) \subset Q_n^{(+\ell)}$ or $Q_n(c) \subset Q_n^{(-\ell)}$, where $Q_n^{(\pm \ell)}$ is the sub-hypercube induced by all vertices $x \in U_n \cup V_n$ with $[x]_\ell = 0$ or $[x]_\ell = 1$ depending on the sign $\pm$. Let us assume $Q_n \subset Q_n^{(+\ell)}$. In this case, $\rho_c$ factors through the surjective map $\pi^{(+\ell)}: C^\ast(Q_n) \to C^\ast(Q_{n-1})$ from Notation \ref{weig::notation:pi_plus_minus} since $\rho_c$ vanishes on all $p_x$ with $[x]_\ell = 0$. Thus, there is an irreducible representation $\sigma: C^\ast(Q_{n-1}) \to B(\mathcal{H})$ with
    \begin{align*}
        \rho_c = \sigma \circ \pi^{(+\ell)}.
    \end{align*}
    By the induction hypothesis there is some $\mathbf{s} \in \Delta_{n-2}$ such that $\sigma$ is unitarily equivalent to a subrepresentation of $\rho_\mathbf{s}$. Consequently, $\rho_c$ is unitarily equivalent to a subrepresentation of $\rho_\mathbf{s} \circ \pi^{(+\ell)}$. However, this is again a subrepresentation of $\rho_\mathbf{t}$ for $\mathbf{t} = [s_0, \dots, s_{\ell-1}, 0, s_\ell, \dots, s_{n-2}] \in \Delta_{n-1}$ by Lemma \ref{reps::lemma:decomposition_of_rho_t_for_boundary_points}.
\end{proof}

\section{Explicit description of hypercube C*-algebras}
\label{sec::explicit_description_of_hypercube_algebras}

Finally, we are ready to describe the hypercube C*-algebras $C^\ast(Q_n)$ explicitly as algebra of continuous functions on the simplex $\Delta_{n-1}$. First, let us introduce a useful notation.

\begin{notation}
    For $n \in \N$ and $\mathbf{t} = [t_0, \dots, t_{n-1}] \in \Delta_{n-1}$ let 
    $$
    Q_n(\mathbf{t}) \subset Q_n
    $$ 
    be the subgraph obtained by removing all edges $ij \in E_n^{(k)}$ with $t_k = 0$, where $E_n^{(k)} = \{ ij \in E_n \mid i = j \# k \}$ as in Notation \ref{pre::notation:subhypercube_Q_n^(k)_and_E_n^(k)}. Evidently, $Q_n(\mathbf{t})$ partitions $Q_n$ into a disjoint union of sub-hypercubes, where $Q_n(\mathbf{t})$ consists of $2^\alpha$-sub-hypercubes of size $2^{n-1-\alpha}$ if $\mathbf{t}$ has $\alpha$ zeros (see the proof of Lemma \ref{hyp::lemma:properties_of_hypercubes2}(1)). 

    We say that a matrix $A \in M_{U_n}$ is in \emph{$\mathbf{t}$-block diagonal form} if it satisfies
    \begin{align*}
        A_{i_1i_2} = 0 \quad \text{whenever } i_1, i_2 \text{ are in distinct connected components of } Q_n(\mathbf{t}).
    \end{align*}
    In this case every block of $A$ corresponds to a connected component of $Q_n(\mathbf{t})$. 
\end{notation}

\begin{example}
    As an example consider the square $Q_2$ as sketched in Figure \ref{reps::ex:admissible_weighting_example}. For $\mathbf{t} = [1,0] = [t, \bar{t}] \in \Delta_1$ the graph $Q_2(\mathbf{t})$ consists of the two horizontal edges. Both connected components contain exactly one vertex from $U_2$ and thus a matrix $A \in M_{U_2} \cong M_2$ is in $\mathbf{t}$-block diagonal form if and only if it is diagonal. 
\end{example}

\begin{lemma}
    \label{expl::lemma:projections_are_block_diagonal}
    For every $\mathbf{t} = [t_0, \dots, t_{n-1}] \in \Delta_{n-1}$ and $x \in U_n \cup V_n$ the projection $\rho_\mathbf{t}(p_x) \in M_{U_n}$ is in $\mathbf{t}$-block diagonal form where $\rho_\mathbf{t}: C^\ast(Q_n) \to M_{U_n}$ is the representation from Definition \ref{weig::def:representation_rho_t}.
\end{lemma}

\begin{proof}
    For $i \in U_n$ we have $\rho_\mathbf{t}(p_i) = E_{ii}$, which is clearly in $\mathbf{t}$-block diagonal form. For $j \in V_n$ recall that $\rho(p_j)$ is the projection onto the span of the vector $\psi_j \in \C^{U_n}$ given by
    \begin{align*}
        [\psi_j]_i = \begin{cases}
            (-1)^{\mathrm{par}_k(i)} \sqrt{t_k}, & \text{ if } j = i \# k \text{ for some } k < n, \\
            0, & \text{ otherwise}.
        \end{cases}
    \end{align*}
    If $i, i^\prime \in U_n$ are in different connected components of $Q_n(\mathbf{t})$, then there are two possibilities: If $j$ is not a common neighbor of $i$ and $i^\prime$ in $Q_n$, then $[\psi_j]_i = 0$ or $[\psi_j]_{i^\prime} = 0$. In any event it follows that the $(i, i^\prime)$-th entry of $\rho_\mathbf{t}(p_j)$ vanishes. If, on the other hand, $j$ is a common neighbor of $i$ and $i^\prime$, then there are $\ell, k < n$ with $j = i \# \ell$ and $j = i^\prime \# k$. Since $i$ and $i^\prime$ are in different connected components of $Q_n(\mathbf{t})$, we must have $t_k = 0$ or $t_\ell = 0$. Again this entails that the $(i, i^\prime)$-th entry of $\rho_\mathbf{t}(p_j)$ vanishes.
    Thus, $\rho_\mathbf{t}(p_j)$ is in $\mathbf{t}$-block diagonal form.
\end{proof}

\begin{thm}
    \label{hyp::thm:hypercube_cstar_algebras_as_continuous_functions}
    The C*-algebra $C^\ast(Q_n)$ is isomorphic to the algebra of continuous functions $f: \Delta_{n-1} \to M_{U_n} \cong M_{2^{n-1}}$ such that $f(\mathbf{t})$ is in $\mathbf{t}$-block diagonal form for all $\mathbf{t} \in \Delta_{n-1}$.
\end{thm}

\begin{proof}
    For every $\mathbf{t} \in \Delta_{n-1}$ and $x \in U_n \cup V_n$ let $p_x^{\mathbf{t}} := \rho_\mathbf{t}(p_x) \in M_{U_n}$. One readily checks that the map $\Delta_{n-1} \to M_{U_n}$ given by $\mathbf{t} \mapsto p_x^{\mathbf{t}}$ is continuous for all $x \in U_n \cup V_n$. Thus, the universal property of $C^\ast(Q_n)$ yields a $\ast$-homomorphism
    \begin{align*}
        \varphi: C^\ast(Q_n) \to C(\Delta_{n-1}, M_{U_n}), \quad
        p_x \mapsto (p_x^{\mathbf{t}})_{\mathbf{t} \in \Delta_{n-1}}.
    \end{align*}
    It remains to show that $\varphi$ is injective and that its image consists of all continuous functions $f: \Delta_{n-1} \to M_{U_n}$ such that $f(\mathbf{t})$ is in $\mathbf{t}$-block diagonal form for all $\mathbf{t} \in \Delta_{n-1}$.

    To show injectivity, it suffices to note that every irreducible representation $\rho$ of $C^\ast(Q_n)$ factors through $\varphi$. Indeed, by Theorem \ref{irreps::thm:classification_of_irreps} every irreducible representation is unitarily equivalent to a subrepresentation of some $\rho_\mathbf{t}$ for $\mathbf{t} \in \Delta_{n-1}$.

    For surjectivity, first observe that for every $\mathbf{t} \in \Delta_{n-1}$ and $x \in U_n \cup V_n$ the projection $p_x^{\mathbf{t}}$ is in $\mathbf{t}$-block diagonal form by Lemma \ref{expl::lemma:projections_are_block_diagonal}.
    On the other hand, whenever $i, i^\prime \in U_n$ are in the same connected component of $Q_n(\mathbf{t})$, then there is a path $i= i_1j_1i_2j_2 \dots j_{m-1} i_m = i^\prime$ in $Q_n(\mathbf{t})$ that connects $i$ and $i^\prime$. Then one has
    \begin{align*}
        p_{i_1}^{\mathbf{t}} p_{j_1}^{\mathbf{t}} p_{i_2}^{\mathbf{t}} p_{j_2}^{\mathbf{t}} \dots p_{j_{m-1}}^{\mathbf{t}} p_{i_m}^{\mathbf{t}} \neq 0,
    \end{align*}
    and the matrix on the left-hand side has a single non-zero entry at the $(i,i^\prime)$-th position. Thus, 
    \begin{align*}
        \varphi(\Delta_{n-1})|_\mathbf{t} := \{ \varphi(x)|_\mathbf{t} \mid x \in C^\ast(Q_n) \} \subset M_{U_n}
    \end{align*}
    is the algebra of matrices in $\mathbf{t}$-block diagonal form. Now, the Stone-Weierstrass theorem for C*-algebras implies that $\varphi$ is surjective, see e.g. \cite[Theorem 11.1.8]{dixmier_c-algebras_1977}.
\end{proof}

\section{Application to magic isometries}
\label{sec::application_to_magic_isometries}

In this final section we apply our findings about the $C^\ast$-algebra of the cube $Q_3$ in order to solve a problem about magic unitaries. We discussed this problem and its background from \cite[Section 5]{banica_noncommutative_2012} in the introduction. Recall the notions of an $n \times m$-magic isometry and an $n \times n$-magic unitary from Definition \ref{intro::def:magic_isometry_unitary}. 

The general question is whether a given $n \times m$-magic isometry can be completed to an $n \times n$-magic unitary by adding further projections. In this section we investigate the particular case of a $2 \times 4$-magic isometry and its completion to a $4 \times 4$-magic unitary. Thus, the problem we solve is the following.

\begin{problem}
    \label{app::problem:2x4_quantum_sudoku_fill_up}
    Can every $2 \times 4$-magic isometry $P = (P_{ij})_{i=1,2; j=1,2,3,4} \subset M_{2,4}(A)$ with entries in a $C^\ast$-algebra $A = C^\ast(P_{ij})$ be filled up to a $4 \times 4$-magic unitary $U = (U_{ij})_{i,j=1}^4 \subset M_4(B)$ with entries in another $C^\ast$-algebra $B$ in the sense that there is an embedding $A \hookrightarrow B$ with $P_{ij} \mapsto U_{ij}$ for $i=1,2$ and $j=1,2,3,4$?
\end{problem}

The following is a simple observation which connects $2 \times 4$-magic isometries to the hypercube $C^\ast$-algebra $C^\ast(Q_3)$.

\begin{proposition}
    \label{app::prop:2x4_magic_isometry_as_hypercube_algebra}
    The universal $C^\ast$-algebra generated by the entries $p_{ij}$ with $i=1,2$ and $j=1,2,3,4$ of a $2 \times 4$-magic isometry is isomorphic to the hypercube C*-algebra $C^\ast(Q_3)$ via
    \begin{align*}
        p_{11} \mapsto p_{000}, & & p_{21} \mapsto p_{111}, \\
        p_{12} \mapsto p_{011}, & & p_{22} \mapsto p_{100}, \\
        p_{13} \mapsto p_{110}, & & p_{23} \mapsto p_{001}, \\
        p_{14} \mapsto p_{101}, & & p_{24} \mapsto p_{010},
    \end{align*}
    where the $p_x$ on the right-hand side are the canonical generators of $C^\ast(Q_3)$ as in Example \ref{hyp::example:representation_rho_t_for_the_cube_Q3}.
\end{proposition}

In what follows we will identify the two C*-algebras from the previous proposition and denote the canonical generators of $C^\ast(Q_3)$ by $p_{ij}$ with $i=1,2$ and $j=1,2,3,4$.

Recall that the universal C*-algebra generated by the entries of a $4 \times 4$-magic unitary $u = (u_{ij})_{i,j=1}^4$ is known as $C(S_4^+)$, the underlying C*-algebra of the quantum permutation group on four points introduced by Wang in \cite{wang_quantum_1998}. This C*-algebra is well-understood, see e.g. \cite{banica_structure_2006,banica_flat_2017,banica_integration_2007,banica_representations_2009,jung_models_2020,katsura_magic_2022}. 

As suggested in \cite{banica_noncommutative_2012}, we will make use of the following faithful representation of $C(S_4^+)$ which is due to Banica and Collins, see also \cite{banica_structure_2006}.

\begin{thm}[{\cite[Theorem 4.1]{banica_integration_2008}}]
    Let 
    \begin{align*}
        c_1 = \begin{pmatrix}
            1 & 0 \\
            0 & 1
        \end{pmatrix}, \;
        c_2 = \begin{pmatrix}
            0 & 1 \\
            1 & 0
        \end{pmatrix}, \;
        c_3 = \begin{pmatrix}
            0 & -i \\
            i & 0
        \end{pmatrix}, \;
        c_4 = \begin{pmatrix}
            1 & 0 \\
            0 & -1
        \end{pmatrix}
    \end{align*}
    be the Pauli matrices, and for every matrix $X \in M_2$ let $\Pi(X) \in M_4$ be the projection onto $\mathrm{span}(X) \subset M_2$ after identifying the Hilbert spaces $M_2$ and $\C^4$. 
    
    The $\ast$-homomorphism $\Phi: C(S_4^+) \to C(\mathrm{SU}(2), M_4)$ given by
    \begin{align*}
        \Phi: u_{ij} \mapsto (\Pi(c_i x c_j))_{x \in \mathrm{SU}(2)} =: (\sigma_x(u_{ij}))_{x \in \mathrm{SU}(2)} \quad \text{ for } i,j \leq 4,
    \end{align*}
    defines a faithful representation of $C(S_4^+)$.
\end{thm}

We can reformulate Problem \ref{app::problem:2x4_quantum_sudoku_fill_up} in terms of universal C*-algebras as follows.

\begin{lemma}
    \label{app::lemma:reformulation_of_quantum_sudoku_fill_up_problem}
    Problem \ref{app::problem:2x4_quantum_sudoku_fill_up} has an affirmative answer if the following condition ($\ast$) holds: Every irreducible representation of $C^\ast(Q_3)$ factors through the unique $\ast$-homomorphism $\varphi: C^\ast(Q_3) \to C(S_4^+)$ with
    \begin{align*}
        \label{hyp::eq:phi_map_from_C*(Q3)_to_C(S4+)}
        \varphi: p_{ij} \mapsto u_{ij} \quad \text{ for all } i \leq 2, j \leq 4,
    \end{align*}
    where we identify the generators of $C^\ast(Q_3)$ with the $p_{ij}$ as in Proposition \ref{app::prop:2x4_magic_isometry_as_hypercube_algebra}. 
    
    This means that whenever $\rho: C^\ast(Q_3) \to B(\mathcal{H})$ is an irreducible representation, then there exists a  representation $\sigma: C(S_4^+) \to B(\mathcal{H})$ such that
    \begin{align*}
        \rho = \sigma \circ \varphi.
    \end{align*}
\end{lemma}

Note that the existence and uniqueness of the map $\varphi$ follows from the universal property of $C(S_4^+)$.

\begin{proof}
    Let $(P_{ij})_{i=1,2; j=1,2,3,4} \subset M_{2,4}(A)$ be a magic isometry where $A = C^\ast(P_{ij})$. The elements $P_{ij}$ satisfy the relations \eqref{bga::eq:partition_relation} and \eqref{bga::eq:orthogonality_relation}; therefore there is a unique (surjective) $\ast$-homomorphism $\pi: C^\ast(Q_3) \to A$ with $\pi(p_{ij}) = P_{ij}$ for $i=1,2$ and $j=1,2,3,4$. Let $I := \mathrm{ker}(\pi)$ be the kernel of $\pi$, and let $J \subset C(S_4^+)$ be the closed ideal generated by $\varphi(I)$ with quotient map $\pi_J: C(S_4^+) \to C(S_4^+) / J$. Since $\pi_J \circ \varphi$ vanishes on $I$, it factors through $\pi$ as in the diagram below.
    \begin{center}
        \begin{tikzcd}
            C^\ast(Q_3) \arrow[r, "\varphi"] \arrow[d, "\pi"'] & C(S_4^+) \arrow[d, "\pi_J"] \\
            A \arrow[r, dashed, "\psi"] & C(S_4^+) / J
        \end{tikzcd}
    \end{center}   
    Evidently, the magic isometry $(P_{ij})_{i \leq 2; j \leq 4}$ can be filled up to the magic unitary $(\pi_J(u_{ij}))_{i,j \leq 4}$ if the map $\psi$ is injective.
 
    To prove that this is true, it suffices to show that every irreducible representation $\rho^\prime$ of $A$ factors through $\psi$. But that is implied by ($*$) as follows. If $\rho^\prime: A \to B(\mathcal{H})$ is irreducible, then the same goes for $\rho := \rho^\prime \circ \pi_I$. By ($*$) that representation factors as $\rho = \sigma \circ \varphi$ for a suitable irreducible representation $\sigma: C(S_4^+) \to B(\mathcal{H})$. As $\rho$ vanishes on $I$, $\sigma$ vanishes on $J$ and therefore factors through $\pi_J$ as $\sigma = \sigma^\prime \circ \pi_J$. This means that $\rho^\prime$ factors through $\psi$ as $\rho^\prime = \sigma^\prime \circ \psi$.
\end{proof}

To prove that condition ($\ast$) in Lemma \ref{app::lemma:reformulation_of_quantum_sudoku_fill_up_problem} is true, we combine our knowledge about $C^\ast(Q_3)$ with the faithful representations of $C(S_4^+)$ due to Banica and Collins mentioned above.
Indeed, all we need is a technical computation.

\begin{lemma}
    \label{hyp::lemma:representation_rho_t_factors_thorugh_sigma_x}
    Let $\mathbf{t} = [a^2,b^2,c^2] \in \Delta_2$ with $a, b, c \geq 0$ be given and set
    \begin{align*}
        x := \begin{pmatrix}
            is & t + iu \\
            -t + iu & -is
        \end{pmatrix}
        \in M_2,
    \end{align*}
    where
    \begin{align*}
        s : = \frac{a}{\sqrt{2}\sqrt{c+1}},
        & & t : = \frac{b}{\sqrt{2}\sqrt{c+1}},
        & & u : = \sqrt{\frac{c+1}{2}}.
    \end{align*}
    Then $x \in \mathrm{SU}(2)$ and $\sigma_x \circ \varphi$ is unitarily equivalent to $\rho_\mathbf{t}$. Here $\sigma_x$ is the representation of $C^\ast(S_4^+)$ from the previous theorem and $\rho_\mathbf{t}$ is the representation of $C^\ast(Q_3)$ induced by the weighting $c_\mathbf{t}$ from Definition \ref{weig::def:representation_rho_t}.
\end{lemma}

\begin{proof}
    For $x \in \mathrm{SU}(2)$ it suffices to check that $s^2+t^2+u^2 = 1$. Indeed,
    \begin{align*}
        s^2+t^2+u^2 = \frac{a^2+b^2}{2(c+1)} + \frac{c+1}{2}
            = \frac{a^2 + b^2 + (c+1)^2}{2(c+1)}
            = \frac{1 + 2c + 1}{2(c+1)} = 1.
    \end{align*}
    Next, recall that $\Pi(c_i x c_j)$ is the projection onto $\mathrm{span}(c_i x c_j)$ after identifying $M_2$ with $\C^4$. The obtained representation $\sigma_x$ is unitarily equivalent to $\sigma_x^\prime$ given by 
    \begin{align*}
        \sigma_x^\prime: u_{ij} \mapsto \Pi(x^\ast c_i x c_j),
    \end{align*}
    for $x^\ast$ is a unitary transformation. 
    Indeed, if $X \in M_4$ is any unitary matrix, then the map $\C^4 \cong M_2 \ni A \mapsto XA \in M_2 \cong \C^4$ is a unitary transformation.
    For the identification of $M_2$ and $\C^4$ we may use the Pauli matrices and identify
    \begin{align*}
        x_1 c_1 + x_2 c_2 + x_3 c_3 + x_4 c_4 = \begin{pmatrix}
            x_1 \\ x_2 \\ x_3 \\ x_4
        \end{pmatrix}.
    \end{align*}
    Then a calculation confirms
    \begin{align*}
        x^\ast c_1 x c_4 &= \begin{pmatrix}
            0 \\ 0 \\ 0 \\ 1
        \end{pmatrix},&
        x^\ast c_1 x c_3 &= \begin{pmatrix}
            0 \\ 0 \\ 1 \\ 0
        \end{pmatrix},&
        x^\ast c_1 x c_2 &= \begin{pmatrix}
            0 \\ 1 \\ 0 \\ 0
        \end{pmatrix}, &
        x^\ast c_1 x c_1 &= \begin{pmatrix}
            1 \\ 0 \\ 0 \\ 0
        \end{pmatrix},  \\
        x^\ast c_2 x c_4 &= \begin{pmatrix}
            a \\ ib \\ -ic \\ 0
        \end{pmatrix}, &
        x^\ast c_2 x c_3 &= \begin{pmatrix}
            b \\ -ia \\ 0 \\ ic
        \end{pmatrix},&
        x^\ast c_2 x c_2 &= \begin{pmatrix}
            c \\ 0 \\ ia \\ -ib
        \end{pmatrix}, &
        x^\ast c_2 x c_1 &= \begin{pmatrix}
            0 \\ c \\ b \\ a
        \end{pmatrix},
    \end{align*}
    using
    \begin{align*}
        c = 2u^2 - 1 = u^2 - s^2 - t^2,
        & & b = 2tu, 
        & & a = 2su.
    \end{align*}
    By applying the unitary transformation $U = \mathrm{diag}(1, -i, i, -i)$ the vectors in the second row are transformed into
    \begin{align*}
        \varphi_{24} &= \begin{pmatrix}
            a \\ b \\ c \\ 0
        \end{pmatrix}, &
        \varphi_{23} &= \begin{pmatrix}
            b \\ -a \\ 0 \\ c
        \end{pmatrix}, &
        \varphi_{22} &= \begin{pmatrix}
            c \\ 0 \\ -a \\ -b
        \end{pmatrix}, &
        \varphi_{21} &= \begin{pmatrix}
            0 \\ -ic \\ ib \\ -ia
        \end{pmatrix}.
    \end{align*}
    Thus, a comparison with the vectors $\psi_j$ for $j \in U_3$ from Example \ref{hyp::example:representation_rho_t_for_the_cube_Q3} immediately yields that $\rho_{\mathbf{t}}$ and $\sigma_x$ are unitarily equivalent.
\end{proof}

\begin{proposition}
    \label{hyp::prop:quantum_sudoku_fill_up}
    The statement from Problem \ref{app::problem:2x4_quantum_sudoku_fill_up} is true.
\end{proposition}

\begin{proof}
    By Lemma \ref{app::lemma:reformulation_of_quantum_sudoku_fill_up_problem} it suffices to show that every irreducible representation $\rho: C^\ast(Q_3) \to B(\mathcal{H})$ factors through $\varphi$.
    However, by Theorem \ref{irreps::thm:classification_of_irreps} every irreducible representation of $C^\ast(Q_3)$ is unitarily equivalent to a subrepresentation of some $\rho_{\mathbf{t}}$ with $\mathbf{t} \in \Delta_2$. Further, by Lemma \ref{hyp::lemma:representation_rho_t_factors_thorugh_sigma_x} every $\rho_{\mathbf{t}}$ factors through $\varphi$. 
\end{proof}

\printbibliography

\end{document}